\documentclass[onefignum,onetabnum]{siamart171218}

\usepackage{tikz}
\usetikzlibrary{arrows.meta, calc, positioning, fit, backgrounds, decorations.pathreplacing, intersections, shapes.geometric}
\definecolor{cudBlue}{HTML}{0072B2}  % #0072B2
\definecolor{cudVermillion}{HTML}{D55E00}  % #D55E00
\definecolor{cudOrange}{HTML}{E69F00}  % #E69F00
\definecolor{cudBluishGreen}{HTML}{009E73}  % #009E73
\definecolor{cudSkyBlue}{HTML}{56B4E9}  % #56B4E9
\definecolor{cudGraphPos}{HTML}{3A94B6}  % #3A94B6
\definecolor{cudGraphNeg}{HTML}{B85C74}  % #B85C74
\definecolor{cudReddishPurple}{HTML}{CC79A7}  % #CC79A7
\definecolor{cudYellow}{HTML}{F0E442}  % #F0E442
\definecolor{cudBlack}{HTML}{000000}  % #000000

\definecolor{cudFrontier}{HTML}{000000}  % #000000

\colorlet{cudBlueDraw}{cudBlue!80!black}  % #0072B2!80!black
\colorlet{cudBlueFill}{cudBlue!12}  % #0072B2!12
\colorlet{cudOrangeDraw}{cudOrange!85!black}  % #E69F00!85!black
\colorlet{cudOrangeFill}{cudOrange!14}  % #E69F00!14
\colorlet{cudVermillionDraw}{cudVermillion!80!black}  % #D55E00!80!black
\colorlet{cudVermillionFill}{cudVermillion!10}  % #D55E00!10
\colorlet{cudGreenDraw}{cudBluishGreen!75!black}  % #009E73!75!black
\colorlet{cudGreenFill}{cudBluishGreen!12}  % #009E73!12
\colorlet{cudGraphPosDraw}{cudGraphPos!80!black}  % #3A94B6!80!black
\colorlet{cudGraphNegDraw}{cudGraphNeg!80!black}  % #B85C74!80!black
\colorlet{cudFrontierDraw}{cudBlack}  % #000000
\colorlet{cudDomainFill}{cudSkyBlue!10}  % #56B4E9!10

\tikzset{
  schottky base/.style={
      font=\small,
      inner sep=2pt,
      minimum height=6mm,
      rounded corners=2pt,
    },
  schottky active/.style={
      schottky base,
      draw=cudBlueDraw,
      fill=cudBlueFill,
      thick,
    },
  schottky faded/.style={
      schottky base,
      draw=gray!55,
      fill=gray!12,
      text=gray!70!black,
      thick,
      dashed,
    },
  schottky kept/.style={
      schottky base,
      draw=cudBlueDraw,
      fill=cudBlueFill,
      thick,
    },
  schottky pruned/.style={
      schottky base,
      draw=gray!55,
      fill=gray!12,
      text=gray!65!black,
      thick,
      dashed,
    },
  schottky missed/.style={
      schottky base,
      draw=cudBlueDraw,
      fill=cudOrangeFill,
      thick,
      dashed,
    },
  schottky path/.style={
      schottky base,
      draw=cudVermillionDraw,
      fill=cudVermillionFill,
      thick,
    },
  schottky edge/.style={
      ->,
      >=Stealth,
      thick,
      draw=cudBlue!75!black,
    },
  schottky edge faded/.style={
      ->,
      >=Stealth,
      thick,
      draw=gray!50,
      dashed,
    },
  schottky edge path/.style={
      ->,
      >=Stealth,
      thick,
      draw=cudVermillionDraw,
    },
  schottky path emph/.style={
      schottky path,
      line width=1.35pt,
    },
  schottky edge path emph/.style={
      schottky edge path,
      line width=1.35pt,
    },
  level box/.style={
      draw=cudBlack,
      dashed,
      thick,
      rounded corners=3pt,
      inner sep=5pt,
    },
  level frontier line/.style={
      draw=cudFrontierDraw,
      dashed,
      ultra thick,
    },
  frontier line/.style={
      level frontier line,
    },
  frontier label/.style={
      font=\footnotesize,
      text=cudFrontierDraw,
    },
  potential frontier label/.style={
      font=\small,
      text=cudFrontierDraw,
    },
  schottky failed frontier/.style={
      draw=gray!55,
      dashed,
      thick,
    },
  schottky strike mark/.style={
      font=\footnotesize\bfseries,
      text=cudVermillionDraw,
    },
  schottky ulabel/.style={
      font=\scriptsize,
      fill=white,
      inner sep=1pt,
      draw=cudVermillion!60!black,
      rounded corners=1pt,
    },
  schottky annot/.style={
      font=\footnotesize,
      align=left,
      text=gray!70!black,
    },
  domain fill/.style={fill=cudDomainFill, draw=black, thick},
  hole fill/.style={fill=white, thick},
  graph node/.style={
      circle,
      draw=cudGraphPosDraw,
      fill=cudGraphPos!10,
      thick,
      minimum size=8mm,
      font=\small,
    },
  graph edge/.style={
      ->,
      >=Stealth,
      thick,
      draw=cudGraphPosDraw,
    },
  graph edge path/.style={
      ->,
      >=Stealth,
      thick,
      draw=cudVermillionDraw,
      dashed,
    },
  cud level label/.style={
      potential frontier label,
    },
  cud note kept/.style={
      font=\scriptsize,
      text=cudBlueDraw,
    },
  schottky prune kept/.style={
      draw=cudBlueDraw,
      fill=cudBlue!22,
      line width=0.9pt,
      inner sep=0.5pt,
      minimum size=4mm,
      rounded corners=1pt,
    },
  schottky prune faded/.style={
      draw=gray!55,
      fill=gray!12,
      line width=0.7pt,
      inner sep=0.5pt,
      minimum size=4mm,
      rounded corners=1pt,
      text=gray!65!black,
      dashed,
    },
  schottky prune psi/.style={
      draw=cudBlueDraw,
      fill=cudBlueFill,
      thick,
      inner sep=1.5pt,
      rounded corners=1.5pt,
    },
  schottky prune tri kept/.style={
      draw=cudBlueDraw,
      fill=cudBlue!8,
      thick,
    },
  schottky prune tri faded/.style={
      draw=gray!55,
      fill=gray!10,
      thick,
      dashed,
    },
  schottky prune edge/.style={
      ->,
      thick,
      draw=cudBlue!75!black,
    },
  schottky prune edge faded/.style={
      ->,
      thick,
      draw=gray!50,
      dashed,
    },
}

\makeatletter

\tikzset{
  schottky ellipsis/.style={
    font=\footnotesize,
    text=gray!60!black,
    inner sep=0pt,
    draw=none,
    fill=none,
  },
}

\newcommand{\stree}[3][]{%
  \draw[schottky edge, #1] (#2.west) -- (#3.east);%
}
\newcommand{\streef}[2]{%
  \draw[schottky edge faded] (#1.west) -- (#2.east);%
}

\def\SchottkyTree@x@id{9.8}
\def\SchottkyTree@x@lone{8.2}
\def\SchottkyTree@x@ldots{6.40}% ellipsis before level~2
\def\SchottkyTree@x@ltwo{5.25}% level~2
\def\SchottkyTree@x@lthreedots{3.30}% ellipsis before level~3
\def\SchottkyTree@x@lthree{1.6}% level~3
\def\SchottkyTree@x@lfourdots{-0.7}% ellipsis before level~4
\def\SchottkyTree@x@lfour{-2.05}% level~4

\newcommand{\SchottkyPlaceId}{%
  \node[schottky active] (id) at (\SchottkyTree@x@id, 0) {$\mathrm{Id}$};%
}

\newcommand{\SchottkyPlaceLOne}{%
  \node[schottky active] (t1) at (\SchottkyTree@x@lone, 2.80) {$\theta_1$};%
  \node[schottky active] (t2) at (\SchottkyTree@x@lone, 1.40) {$\theta_2$};%
  \node[schottky active] (tm1) at (\SchottkyTree@x@lone, 0.00) {$\theta_{-1}$};%
  \node[schottky active] (tm2) at (\SchottkyTree@x@lone, -1.40) {$\theta_{-2}$};%
}

\newcommand{\SchottkyPlaceEllipsis}{%
  \node[schottky ellipsis] (dots2) at (\SchottkyTree@x@ldots, 1.40) {$\cdots$};%
  \node[schottky ellipsis] (dotsm1) at (\SchottkyTree@x@ldots, 0.00) {$\cdots$};%
}

\newcommand{\SchottkyPlaceLTwoFromTone}{%
  \node[schottky active] (t1t1) at (\SchottkyTree@x@ltwo, 3.30) {$\theta_1\theta_1$};%
  \node[schottky active] (t2t1) at (\SchottkyTree@x@ltwo, 2.20) {$\theta_2\theta_1$};%
  \node[schottky active] (tm2t1) at (\SchottkyTree@x@ltwo, 1.10) {$\theta_{-2}\theta_1$};%
}

\newcommand{\SchottkyPlaceLTwoFromTmtwo}{%
  \node[schottky active] (t1m2) at (\SchottkyTree@x@ltwo, -0.40) {$\theta_1\theta_{-2}$};%
  \node[schottky active] (tm1m2) at (\SchottkyTree@x@ltwo, -1.50) {$\theta_{-1}\theta_{-2}$};%
  \node[schottky active] (tm2m2) at (\SchottkyTree@x@ltwo, -2.60) {$\theta_{-2}\theta_{-2}$};%
}

\newcommand{\SchottkyPlaceLThreeFromTmTwoTone}{%
  \node[schottky active] (t1m2t1) at (\SchottkyTree@x@lthree, 2.60) {$\theta_1\theta_{-2}\theta_1$};%
  \node[schottky active] (tm1m2t1) at (\SchottkyTree@x@lthree, 1.50) {$\theta_{-1}\theta_{-2}\theta_1$};%
  \node[schottky active] (tm2m2t1) at (\SchottkyTree@x@lthree, 0.40) {$\theta_{-2}\theta_{-2}\theta_1$};%
}

\newcommand{\SchottkyPlaceLThreeFromTmOneTmTwo}{%
  \node[schottky active] (t2tm1m2) at (\SchottkyTree@x@lthree, -0.95) {$\theta_2\theta_{-1}\theta_{-2}$};%
}

\newcommand{\SchottkyPlaceLThreeDots}{%
  \node[schottky ellipsis] (dotsT1T1) at (\SchottkyTree@x@lthreedots, 3.30) {$\cdots$};%
  \node[schottky ellipsis] (dotsT2T1) at (\SchottkyTree@x@lthreedots, 2.20) {$\cdots$};%
  \node[schottky ellipsis] (dotsT1M2) at (\SchottkyTree@x@lthreedots, -0.40) {$\cdots$};%
  \node[schottky ellipsis] (dotsM2M2) at (\SchottkyTree@x@lthreedots, -2.60) {$\cdots$};%
  \node[schottky ellipsis] (dotsTm1M2) at (\SchottkyTree@x@lthreedots, -2.05) {$\cdots$};%
}

\newcommand{\SchottkyPlaceLFourDots}{%
  \node[schottky ellipsis] (dotsL3a) at (\SchottkyTree@x@lfourdots, 2.60) {$\cdots$};%
  \node[schottky ellipsis] (dotsL3b) at (\SchottkyTree@x@lfourdots, 1.50) {$\cdots$};%
  \node[schottky ellipsis] (dotsL4) at (\SchottkyTree@x@lfourdots, -0.35) {$\cdots$};%
  \node[schottky ellipsis] (dotsL3c) at (\SchottkyTree@x@lfourdots, -0.95) {$\cdots$};%
}

\newcommand{\SchottkyPlaceLFour}{%
  \node[schottky active] (tm1m2m2t1) at (\SchottkyTree@x@lfour, 0.40) {$\theta_{-1}\theta_{-2}\theta_{-2}\theta_1$};%
}

\newcommand{\SchottkyPlaceLFourTruncated}{%
  \SchottkyPlaceLFourDots
  \node[schottky pruned] (tm1m2m2t1) at (\SchottkyTree@x@lfour, 0.40) {$\theta_{-1}\theta_{-2}\theta_{-2}\theta_1$};%
}

\newcommand{\SchottkyPlaceLTwo}{%
  \SchottkyPlaceLTwoFromTone
  \SchottkyPlaceLTwoFromTmtwo
}

\newcommand{\SchottkyPlaceLThree}{%
  \SchottkyPlaceLThreeFromTmTwoTone
  \SchottkyPlaceLThreeFromTmOneTmTwo
  \SchottkyPlaceLThreeDots
}

\newcommand{\SchottkyEdgesToLOne}{%
  \stree{id}{t1}%
  \stree{id}{t2}%
  \stree{id}{tm1}%
  \stree{id}{tm2}%
}

\newcommand{\SchottkyEdgesEllipsis}{%
  \stree{t2}{dots2}%
  \stree{tm1}{dotsm1}%
}

\newcommand{\SchottkyEdgesLTwoFromTone}{%
  \stree{t1}{t1t1}%
  \stree{t1}{t2t1}%
  \stree{t1}{tm2t1}%
}

\newcommand{\SchottkyEdgesLTwoFromTmtwo}{%
  \stree{tm2}{t1m2}%
  \stree{tm2}{tm1m2}%
  \stree{tm2}{tm2m2}%
}

\newcommand{\SchottkyEdgesLThreeFromTmTwoTone}{%
  \stree{tm2t1}{t1m2t1}%
  \stree{tm2t1}{tm1m2t1}%
  \stree{tm2t1}{tm2m2t1}%
}

\newcommand{\SchottkyEdgesLThreeFromTmOneTmTwo}{%
  \stree{tm1m2}{t2tm1m2}%
  \stree{tm1m2}{dotsTm1M2}%
}

\newcommand{\SchottkyEdgesLThreeDots}{%
  \stree{t1t1}{dotsT1T1}%
  \stree{t2t1}{dotsT2T1}%
  \stree{t1m2}{dotsT1M2}%
  \stree{tm2m2}{dotsM2M2}%
}

\newcommand{\SchottkyEdgesLFourDots}{%
  \stree{t1m2t1}{dotsL3a}%
  \stree{tm1m2t1}{dotsL3b}%
  \stree{tm2m2t1}{dotsL4}%
  \stree{t2tm1m2}{dotsL3c}%
}

\newcommand{\SchottkyEdgesLFourDotsFaded}{%
  \streef{t1m2t1}{dotsL3a}%
  \streef{tm1m2t1}{dotsL3b}%
  \streef{tm2m2t1}{dotsL4}%
  \streef{t2tm1m2}{dotsL3c}%
}

\newcommand{\SchottkyEdgesLFour}{%
  \stree{tm2m2t1}{tm1m2m2t1}%
}

\newcommand{\SchottkyEdgesLFourFaded}{%
  \streef{tm2m2t1}{tm1m2m2t1}%
}

\newcommand{\SchottkyTreeLmaxThree}{%
  \SchottkyPlaceId
  \SchottkyPlaceLOne
  \SchottkyPlaceEllipsis
  \SchottkyPlaceLTwo
  \SchottkyPlaceLThree
  \SchottkyPlaceLFourTruncated
  \SchottkyEdgesToLOne
  \SchottkyEdgesEllipsis
  \SchottkyEdgesLTwo
  \SchottkyEdgesLThree
  \SchottkyEdgesLFourDotsFaded
  \SchottkyEdgesLFourFaded
}

\newcommand{\SchottkyTreeMTwoLevelPartition}{%
  \SchottkyPlaceId
  \SchottkyPlaceLOne
  \SchottkyPlaceEllipsis
  \SchottkyPlaceLTwo
  \SchottkyPlaceLThree
  \SchottkyPlaceLFourDots
  \SchottkyPlaceLFour
  \SchottkyEdgesToLOne
  \SchottkyEdgesEllipsis
  \SchottkyEdgesLTwo
  \SchottkyEdgesLThree
  \SchottkyEdgesLFourDots
  \SchottkyEdgesLFour
}

\newcommand{\SchottkyEdgesLTwo}{%
  \SchottkyEdgesLTwoFromTone
  \SchottkyEdgesLTwoFromTmtwo
}

\newcommand{\SchottkyEdgesLThree}{%
  \SchottkyEdgesLThreeFromTmTwoTone
  \SchottkyEdgesLThreeFromTmOneTmTwo
  \SchottkyEdgesLThreeDots
}

\newcommand{\SchottkyDrawTreeThetaOnePath}{%
  \node[schottky active] (id) at (\SchottkyTree@x@id, 2.80) {$\mathrm{Id}$};%
  \node[schottky path emph] (t1) at (\SchottkyTree@x@lone, 2.80) {$\theta_1$};%
  \node[schottky active] (t1t1) at (\SchottkyTree@x@ltwo, 3.30) {$\theta_1\theta_1$};%
  \node[schottky active] (t2t1) at (\SchottkyTree@x@ltwo, 2.20) {$\theta_2\theta_1$};%
  \node[schottky path emph] (tm2t1) at (\SchottkyTree@x@ltwo, 1.10) {$\theta_{-2}\theta_1$};%
  \node[schottky active] (t1m2t1) at (\SchottkyTree@x@lthree, 2.60) {$\theta_1\theta_{-2}\theta_1$};%
  \node[schottky path emph] (tm1m2t1) at (\SchottkyTree@x@lthree, 1.50) {$\theta_{-1}\theta_{-2}\theta_1$};%
  \node[schottky active] (tm2m2t1) at (\SchottkyTree@x@lthree, 0.40) {$\theta_{-2}\theta_{-2}\theta_1$};%
  \stree{id}{t1}%
  \stree{t1}{t1t1}%
  \stree{t1}{t2t1}%
  \path (t1.west) edge[schottky edge path emph] node[above, anchor=west, font=\scriptsize, pos=0.62] {$\Delta_{-2}(\theta_1)$} (tm2t1.east);%
  \stree{tm2t1}{t1m2t1}%
  \path (tm2t1.west) edge[schottky edge path emph] node[above, font=\scriptsize, pos=0.48, xshift=4pt, yshift=1pt] {$\Delta_{-1}(\theta_{-2}\theta_1)$} (tm1m2t1.east);%
  \stree{tm2t1}{tm2m2t1}%
}

\newcommand{\Schottky@placePotentialTree}[4]{%
  \SchottkyPlaceId
  \node[schottky kept] (t1) at (\SchottkyTree@x@lone, 2.80) {$\theta_1$};%
  \node[schottky kept] (t2) at (\SchottkyTree@x@lone, 1.40) {$\theta_2$};%
  \node[schottky kept] (tm1) at (\SchottkyTree@x@lone, 0.00) {$\theta_{-1}$};%
  \node[schottky kept] (tm2) at (\SchottkyTree@x@lone, -1.40) {$\theta_{-2}$};%
  \SchottkyPlaceEllipsis
  \node[schottky kept] (t1t1) at (\SchottkyTree@x@ltwo, 3.30) {$\theta_1\theta_1$};%
  \node[schottky pruned] (t2t1) at (\SchottkyTree@x@ltwo, 2.20) {$\theta_2\theta_1$};%
  \node[schottky pruned] (tm2t1) at (\SchottkyTree@x@ltwo, 1.10) {$\theta_{-2}\theta_1$};%
  \node[schottky kept] (t1m2) at (\SchottkyTree@x@ltwo, -0.40) {$\theta_1\theta_{-2}$};%
  \node[schottky pruned] (tm1m2) at (\SchottkyTree@x@ltwo, -1.50) {$\theta_{-1}\theta_{-2}$};%
  \node[schottky pruned] (tm2m2) at (\SchottkyTree@x@ltwo, -2.60) {$\theta_{-2}\theta_{-2}$};%
  \node[#1] (t1m2t1) at (\SchottkyTree@x@lthree, 2.60) {$\theta_1\theta_{-2}\theta_1$};%
  \node[#2] (tm1m2t1) at (\SchottkyTree@x@lthree, 1.50) {$\theta_{-1}\theta_{-2}\theta_1$};%
  \node[schottky pruned] (tm2m2t1) at (\SchottkyTree@x@lthree, 0.40) {$\theta_{-2}\theta_{-2}\theta_1$};%
  \node[#3] (t2tm1m2) at (\SchottkyTree@x@lthree, -0.95) {$\theta_2\theta_{-1}\theta_{-2}$};%
  \SchottkyPlaceLThreeDots
  \SchottkyPlaceLFourDots
  \node[#4] (tm1m2m2t1) at (\SchottkyTree@x@lfour, 0.40) {$\theta_{-1}\theta_{-2}\theta_{-2}\theta_1$};%
}

\newcommand{\SchottkyPlacePotentialExample}{%
  \Schottky@placePotentialTree{schottky kept}{schottky kept}{schottky kept}{schottky kept}%
}

\newcommand{\SchottkyPlacePotentialLevelMissed}{%
  \Schottky@placePotentialTree{schottky kept}{schottky kept}{schottky kept}{schottky missed}%
}

\newcommand{\SchottkyPlacePotentialCurveMissed}{%
  \Schottky@placePotentialTree{schottky missed}{schottky missed}{schottky missed}{schottky missed}%
}

\newcommand{\SchottkyPlacePotentialCurveComplete}{%
  \Schottky@placePotentialTree{schottky kept}{schottky kept}{schottky kept}{schottky kept}%
}

\newcommand{\SchottkyEdgesPotentialExample}{%
  \SchottkyEdgesToLOne
  \SchottkyEdgesEllipsis
  \SchottkyEdgesLTwo
  \SchottkyEdgesLThree
  \SchottkyEdgesLFourDots
  \SchottkyEdgesLFour
}

\edef\SchottkyCoordId{\SchottkyTree@x@id}
\edef\SchottkyCoordLone{\SchottkyTree@x@lone}
\edef\SchottkyCoordLdots{\SchottkyTree@x@ldots}
\edef\SchottkyCoordLtwo{\SchottkyTree@x@ltwo}
\edef\SchottkyCoordLthreeDots{\SchottkyTree@x@lthreedots}
\edef\SchottkyCoordLthree{\SchottkyTree@x@lthree}
\edef\SchottkyCoordLfourDots{\SchottkyTree@x@lfourdots}
\edef\SchottkyCoordLfour{\SchottkyTree@x@lfour}
\edef\SchottkyCoordFrontierLmaxThree{-0.05}

\makeatother

\usepackage{amsfonts}
\usepackage{amsmath}
\usepackage{amssymb}
\usepackage{amsopn}
\usepackage{graphicx}
\usepackage{epstopdf}
\usepackage[caption=false]{subfig}
\crefname{subfigure}{Figure}{Figures}
\Crefname{subfigure}{Figure}{Figures}
\usepackage{algorithm}
\usepackage{algorithmic}
\usepackage{mathtools}
\usepackage{physics2}
\usephysicsmodule{ab}
\usepackage{diffcoeff}
\usepackage{xcolor}
\usepackage{comment}

\definecolor{cudBlue}{HTML}{0072B2}
\newcommand{\tikzlegendkept}{\textcolor{cudBlue!80!black}{blue}}
\newcommand{\tikzlegendpruned}{\textcolor{gray!65!black}{gray}}

\newif\ifshowdraft
\showdraftfalse

\newcommand{\aihint}[1]{}
\excludecomment{aimark}

\newcommand{\mynote}[1]{}

\ifpdf
  \DeclareGraphicsExtensions{.eps,.pdf,.png,.jpg}
\else
  \DeclareGraphicsExtensions{.eps}
\fi

\graphicspath{{fig/}}

\newsiamremark{remark}{Remark}
\newsiamremark{hypothesis}{Hypothesis}
\crefname{hypothesis}{Hypothesis}{Hypotheses}
\newsiamthm{claim}{Claim}

\DeclareMathOperator{\Card}{Card}
\DeclareMathOperator{\level}{level}
\newcommand{\lb}{\mathrm{lb}}
\newcommand{\ub}{\mathrm{ub}}
\newcommand{\fix}{\mathrm{fix}}
\newcommand{\Real}[1]{\mathrm{Re}[#1]}
\newcommand{\Id}{\mathrm{Id}}

\newcommand{\Complex}{\mathbb{C}}
\newcommand{\Reals}{\mathbb{R}}
\newcommand{\UpperHalf}{\mathbb{H}}
\newcommand{\Domain}{D_\zeta}

\newcommand{\iunit}{\mathrm{i}}
\newcommand{\conj}[1]{\overline{#1}}
\newcommand{\abs}[1]{\ab|#1|}
\newcommand{\closure}[1]{\overline{#1}}

\newcommand{\defeq}{\coloneqq}

\headers{Computing SK-Prime Products}%
{S. Yamamoto and H. Miyoshi}

\title{Fast computation and convergence analysis of the infinite-product representation of the Schottky--Klein prime function%
  \thanks{% Submitted to the editors \draftinline{DATE}.
    {\bfseries Funding:} This work was supported by a JSPS Postdoctoral Fellowship (Grant Number JP24KJ0041) and JSPS KAKENHI Grant Number JP26K17030.}}

\author{Shuntaro Yamamoto\thanks{%
    Department of Mathematical Informatics, Graduate School of Information Science and Technology, The University of Tokyo, Tokyo, Japan (\email{shun0923@g.ecc.u-tokyo.ac.jp}, \email{hiroyukimiyoshi@g.ecc.u-tokyo.ac.jp}).}
  \and Hiroyuki Miyoshi\footnotemark[2]}

\makeatletter
\newcommand*{\addFileDependency}[1]{%
  \typeout{(#1)}%
  \@addtofilelist{#1}%
  \IfFileExists{#1}{}{\typeout{No file #1.}}%
}

\newcommand*{\myexternaldocument}[1]{%
  \IfFileExists{xr-#1.aux}{%
    \externaldocument[{}][#1.pdf]{xr-#1}%
    \addFileDependency{xr-#1.aux}%
  }{%
    \IfFileExists{#1.aux}{%
      \externaldocument{#1}%
      \addFileDependency{#1.aux}%
    }{%
      \externaldocument{#1}%
      \addFileDependency{#1.aux}%
    }%
  }%
  \addFileDependency{#1.tex}%
}

\def\Hy@arxiv@register@linkcolors{%
  \AtBeginDocument{\hypersetup{urlcolor=siaminlinkcolor}}%
  \pretocmd{\bibliography}{%
    \hypersetup{urlcolor=siamexlinkcolor}%
  }{}{}%
}
\IfFileExists{xr-main.aux}{%
  \Hy@arxiv@register@linkcolors
}{%
  \IfFileExists{xr-supplement.aux}{%
    \Hy@arxiv@register@linkcolors
  }{}%
}
\makeatother
\usepackage{booktabs}

\ifpdf
\hypersetup{
  pdftitle={Fast computation and convergence analysis of the infinite-product representation of the Schottky--Klein prime function},
  pdfauthor={S. Yamamoto and H. Miyoshi}
}
\fi

\myexternaldocument{supplement}

\begin{document}

\maketitle

% -------------------- Abstract --------------------
\begin{abstract}
  The Schottky--Klein prime function is a standard tool for boundary-value problems on multiply connected circular domains.
  Because this function is represented as an infinite product over a Schottky group, numerical evaluation requires truncation to finitely many factors.
  The standard word-length truncation grows exponentially in cost and becomes inefficient when the boundary circles nearly touch one another or the unit circle.
  To address this difficulty, we assign to each group element a cross-ratio \emph{potential} measuring the size of its contribution, and retain only terms below a prescribed threshold.
  We establish uniform closed-form bounds on the change in this potential when prepending Schottky-group generators, and from these bounds we derive an efficient enumeration algorithm.
  The resulting relative error decays exponentially with the threshold at a rate determined by the Hausdorff dimension of the limit set of the Schottky group.
  Numerical experiments demonstrate that the proposed formulation achieves substantial computational speedups over word-length truncation in challenging geometric configurations.
\end{abstract}

% -------------------- Keywords / AMS --------------------
\begin{keywords}
  Schottky--Klein prime function, multiply connected domains,
  Schottky group, numerical evaluation, hyperbolic geometry, Hausdorff dimension
\end{keywords}

\begin{AMS}
  30F40, 65E05, 30F45
\end{AMS}

% ============================================================
\section{Introduction}
\label{sec:intro}

\subsection{Multiply connected domains and the Schottky--Klein prime function}
Many boundary-value problems in fluid mechanics, electrostatics, and heat conduction are naturally posed on two-dimensional \emph{multiply connected} domains, that is, domains with finitely many disjoint holes.
Examples include the motion of point vortices around several circular islands~\cite{crowdy2005Motion}, steadily rotating arrays of vortex patches~\cite{crowdy2005Analytical}, and source/sink flows in circular domains~\cite{crowdy2013Analytical}.
In these problems, the analytical solution can be expressed in terms of the \emph{Schottky--Klein prime function} (S--K prime function)~$\omega(\zeta, \alpha)$, a special function introduced by Schottky~\cite{schottky1887Ueber} and studied in the classical theory of Schottky groups and Abelian functions~\cite{klein1890Zur,baker1897Abels}.
More recently, Crowdy and collaborators developed this function into a computational tool for applications in multiply connected domains~\cite{crowdy2007Computing,crowdy2020Solving}.
The S--K prime function provides a common building block for these problems: once $\omega$ has been evaluated numerically on circular domains, the analytical solutions built from it become numerically accessible.
Recent applications include channel and fin-array flows~\cite{miyoshi2022Longitudinal,miyoshi2024Fully}, doubly periodic vortex dynamics~\cite{krishnamurthy2023Nvortex}, nematic alignment on multiply connected domains~\cite{miyoshi2025Analytical}, and harmonic-measure distribution functions for slit domains~\cite{green2022Harmonicmeasure}.
Extensions of the Schwarz--Christoffel theorem to multiply connected and periodic domains remain active topics of research building on this function theory~\cite{crowdy2005Schwarz,baddoo2019Periodic}.

\subsection{Infinite-product representation and level truncation}
\label{sec:intro-limitations}
On multiply connected circular domains with $M \ge 1$ holes, the S--K prime function is represented as the infinite product
\begin{equation}
  \omega(\zeta, \alpha)
  = (\zeta - \alpha)\prod_{\theta \in \Theta''}
  R_{\zeta, \alpha}(\theta),
  \qquad
  R_{\zeta, \alpha}(\theta) \defeq
  \frac{\ab(\theta(\zeta) - \alpha)\ab(\theta(\alpha) - \zeta)}
  {\ab(\theta(\zeta) - \zeta)\ab(\theta(\alpha) - \alpha)},
  \label{eq:sk-prime-intro}
\end{equation}
over the set $\Theta''$ of all non-identity elements of the Schottky group~$\Theta$, with inverses omitted; see \cref{sec:skprime} for the precise definition.

It is not known whether the infinite product in \eqref{eq:sk-prime-intro} converges for every multiply connected circular domain~\cite{baker1897Abels,burnside1890Functions,crowdy2007Computing,belokolos1994Algebrogeometric}.
Classical treatments impose geometric restrictions on the Schottky group, such as \emph{circle-decomposable} domains~\cite{bobenko2011Introduction,belokolos1994Algebrogeometric}, and a complete convergence theory for general configurations remains open~\cite{crowdy2008Geometric}.
When the holes are well separated from one another and from the unit circle, the product is expected to converge in numerical applications~\cite{crowdy2016Schottky,crowdy2020Solving}, but this remains a heuristic observation without a quantitative rate.

In numerical computations, only finitely many factors can be evaluated, so the product~\eqref{eq:sk-prime-intro} must be truncated.
The classical \emph{level truncation} retains all group elements up to a prescribed word length~$L_{\max}$.
The resulting partial product $\omega_{L_{\max}} \defeq(\zeta - \alpha)\prod_{\theta \in \Theta''_{\le L_{\max}}} R_{\zeta, \alpha}(\theta)$ is the standard tool implemented in essentially all product-based S--K prime codes.

There are, however, two well-known difficulties with level truncation.
First, the number of level-$\ell$ elements grows as $M(2M - 1)^{\ell - 1}$, so the cost of computing $\omega_{L_{\max}}$ is exponential in~$L_{\max}$.
Second, even when the full product converges, level truncation converges very slowly in challenging geometries.
Words such as $\theta_1^n$ when a hole lies close to the unit circle or $(\theta_1\theta_{-2})^n$ when two holes are closely packed may have a much smaller per-term decay than a much longer word built from well-separated generators.
Therefore, level ordering does not correlate well with the actual term sizes.

\subsection{Prior work beyond level truncation}
\label{sec:intro-bvp}
A complementary computational route evaluates $\omega$ without truncating the product~\eqref{eq:sk-prime-intro}.
Crowdy and Marshall~\cite{crowdy2007Computing} developed a Fourier--Laurent scheme for evaluating $\omega$ on circular domains via the Schottky double.
A decade later, Crowdy, Kropf, Green, and Nasser~\cite{crowdy2016Schottky} presented boundary-value schemes for evaluating the S--K prime function on circular domains, with a public MATLAB implementation in the SKPrime package~\cite{kropf2016SKPrime}.
This boundary-value approach is often efficient in practice, but its truncation parameter resides on the Fourier side of the formulation.
The resulting pointwise error in~$\omega$ is therefore not directly tied to the Schottky-group product~\eqref{eq:sk-prime-intro}, and \emph{a priori} error bounds are correspondingly harder to derive.
We therefore retain the infinite-product representation and seek truncations for which the omitted factors admit explicit asymptotic error analysis.

Within the product representation, Pollicott~\cite{pollicott2021Schottky} analyzed~\eqref{eq:sk-prime-intro} using a hyperbolic geometry, and obtained a counting theorem for factors of a given size and an asymptotic relating term size to hyperbolic distance.
For fixed $\zeta, \alpha$ in the domain, his counting theorem (\cref{thm:pollicott}) asserts that
$\Card \ab\{\theta \in \Theta : \abs{R_{\zeta, \alpha}(\theta) - 1} \ge 1/T\} \sim C T^{\delta}$ as $T \to \infty$, where $\delta$ is the Hausdorff dimension of the limit set of~$\Theta$.
Together with the asymptotic $\abs{R_{\zeta, \alpha}(\theta) - 1} \sim 4 e^{-d(\theta)}$, this identifies distance-based ordering, rather than word length, as the natural truncation parameter for the S--K prime product.
In particular, when the factors are ordered by decreasing $\abs{R_{\zeta, \alpha}(\theta) - 1}$, the product converges whenever $\delta < 1$~\cite{pollicott2021Schottky}.

\subsection{Contributions}
\label{sec:contrib}
Building on this distance-based ordering, we develop a potential truncation of the infinite-product representation together with an efficient enumeration algorithm.
We introduce a potential that ranks product factors by size, and we truncate the product to those factors whose potential is at most a cutoff~$U_{\max}$.
Unlike level truncation, whether a factor is kept depends on~$(\zeta, \alpha)$ and the hole configuration rather than on word length, so the truncation typically cuts across levels of the tree of the Schottky group.
We prove a \emph{uniform increment bound} on the change in potential upon prepending a generator (\cref{thm:increment}), which provides the lower bounds used as the pruning criterion when a depth-first search enumerates the surviving factors.
We also show that the relative error decays exponentially in~$U_{\max}$ at rate~$1 - \delta$ whenever $\delta < 1$ (\cref{thm:error}), where $\delta$ is the Hausdorff dimension of the limit set of~$\Theta$, and that the same counting estimate controls the complexity of the enumeration (\cref{prop:alg}).

Numerical experiments on three representative configurations show that potential truncation outperforms level truncation, most clearly when the holes are closely packed, and that the advantage persists globally when~$\zeta$ varies over~$\Domain$ at fixed~$\alpha$.
We also compare with SKPrime~\cite{kropf2016SKPrime,crowdy2016Schottky} as an external baseline.

\Cref{sec:prelim} reviews the Schottky group, the S--K prime function, and Pollicott's counting theorem.
The potential truncation and enumeration algorithm are developed in \cref{sec:potential,sec:algorithm,sec:increment,sec:error}, along with increment and error analyses.
Numerical experiments appear in \cref{sec:experiments}; the detailed proofs, the closed-form edge-weight formulas, and the Bellman--Ford algorithm for computing the minimum walk weight are presented in the supplementary material.

% ============================================================
\section{Preliminaries}
\label{sec:prelim}

\subsection{Multiply connected domains and the Schottky group}
\label{sec:schottky}

We follow the standard circular-domain formulation of Crowdy~\cite{crowdy2020Solving}.
Fix an integer $M \ge 1$.
The \emph{unit disk} is
\begin{equation}
  D_0 = \{\zeta \in \Complex : \abs{\zeta} < 1\},
  \qquad
  C_0 = \partial D_0 = \{\zeta \in \Complex : \abs{\zeta} = 1\},
  \label{eq:D0}
\end{equation}
where the boundary $C_0$ is the unit circle.
For each $m = 1, \ldots, M$, let $D_m$ denote the open disk
\begin{equation}
  D_m = \{\zeta \in \Complex : \abs{\zeta - \delta_m} < q_m\},
  \qquad
  C_m = \partial D_m,
  \label{eq:Dm}
\end{equation}
with center $\delta_m \in \Complex$ and radius $q_m > 0$.
We assume that the holes are pairwise disjoint and lie in the interior of the unit disk,
$\closure{D_i} \cap \closure{D_j} = \emptyset$ for all $i\neq j$,
and $\closure{D_i}\subset D_0$ for $i = 1, \ldots, M$.
The \emph{multiply connected domain} is then
\begin{equation}
  \Domain
  = D_0 \setminus \bigcup_{m = 1}^{M}\closure{D_m},
  \label{eq:domain-D}
\end{equation}
where $\Domain$ is the unit disk with $M$ circular holes removed.

Reflection in the unit circle $C_0$ is $\zeta \mapsto 1/\conj{\zeta}$, and reflection in $C_m$ is $\zeta \mapsto \delta_m + q_m^2/(\conj{\zeta} - \conj{\delta_m})$.
Composing these two reflections defines the \emph{basic M\"obius transformations}
\begin{equation}
  \theta_m(\zeta)
  = \delta_m + \frac{q_m^2\zeta}{1 - \conj{\delta_m}\zeta},
  \qquad m = 1, \ldots, M.
  \label{eq:theta-m}
\end{equation}
The inverses are written $\theta_{-m} = \theta_m^{-1}$ and are given by
\begin{equation}
  \theta_{-m}(\zeta)
  = \frac{\zeta - \delta_m}{q_m^2 - \abs{\delta_m}^2 + \conj{\delta_m}\zeta},
  \qquad m = 1, \ldots, M.
  \label{eq:theta-minus-m}
\end{equation}
For each $m = 1, \ldots, M$, the circles $C_m$ and $C_{-m}$ are related by
\begin{equation}
  C_{-m} = \theta_{-m}(C_m),
  \qquad
  C_m = \theta_m(C_{-m}).
  \label{eq:C-minus}
\end{equation}
Let $\delta_{-m}$ and $q_{-m}$ denote the center and radius of~$C_{-m}$:
\begin{equation}
  \delta_{-m}
  = \frac{\delta_m}{\abs{\delta_m}^2 - q_m^2},
  \qquad
  q_{-m}
  = \frac{q_m}{\abs{\abs{\delta_m}^2 - q_m^2}}.
  \label{eq:C-minus-params}
\end{equation}
Define
\begin{equation}
  D_{-m}
  \defeq
  \theta_{-m}\ab(\Complex \setminus \closure{D_m}),
  \label{eq:D-minus}
\end{equation}
so that $C_{-m} = \partial D_{-m}$.
Then, $\theta_m$ maps $D_{-m}$ onto $\Complex \setminus \closure{D_m}$ and $\Complex \setminus \closure{D_{-m}}$ onto $D_m$, whereas $\theta_{-m}$ maps $D_m$ onto $\Complex \setminus \closure{D_{-m}}$ and $\Complex \setminus \closure{D_m}$ onto $D_{-m}$.
For $j \in \{\pm 1, \ldots, \pm M\}$, we write $D_{-j}$ and $C_{-j}$ with the convention $D_{-(-m)} = D_m$ and $C_{-(-m)} = C_m$.

\begin{definition}[Schottky group~\cite{crowdy2020Solving}]
  \label{def:schottky}
  The \emph{Schottky group} $\Theta$ is the free group generated by the basic M\"obius transformations $\theta_1, \ldots, \theta_M$:
  \begin{equation}
    \begin{aligned}
      \Theta
       & \defeq
      \ab<\theta_1, \ldots, \theta_M>
      \\
       & =
      \{\Id\}
      \cup
      \ab\{
      \theta_{j_\ell} \cdots \theta_{j_1}
      :
      \ell \ge 1,\
      j_i \in \{\pm 1, \ldots, \pm M\},\,
      j_{i + 1} \neq -j_i
      \ (1 \le i < \ell)
      \}
      \\
       & =
      \ab\{
      \Id,\,
      \theta_1,\, \theta_2,\, \ldots,\, \theta_{-1},\, \ldots,\,
      \theta_1\theta_2,\, \ldots,\,
      \theta_1\theta_{-2}\theta_3,\, \ldots
      \},
    \end{aligned}
    \label{eq:schottky-group}
  \end{equation}
  where $\theta_{-m} = \theta_m^{-1}$.
  Every non-identity element admits a unique factorization of this form, written $\theta = \theta_{j_\ell} \cdots \theta_{j_1}$ and called a \emph{reduced word}; we call $\ell$ the \emph{level} (or word length) of~$\theta$, written $\level(\theta) = \ell$.
  The identity is the only element of level~$0$.
  In the word $\theta_{j_\ell} \cdots \theta_{j_1}$, the rightmost factor acts first on~$z$: the map is $\theta_{j_\ell}\circ \cdots \circ\theta_{j_1}$, with $(\theta_{j_\ell} \cdots \theta_{j_1})(z) = \theta_{j_\ell}( \cdots(\theta_{j_1}(z)) \cdots)$.
  We write $\theta\psi$ for the composition $\theta\circ\psi$.
\end{definition}
\begin{definition}[Quotient set $\Theta''$]
  \label{def:Theta-double-prime}
  The set $\Theta''$ comprises all elements of the Schottky group $\Theta$ except the identity and all inverses: if the term associated with $\theta \in \Theta$ appears in the infinite product~\eqref{eq:sk-prime-intro}, then the term for $\theta^{-1}$ is omitted.
  Equivalently,
  \begin{equation}
    \Theta''
    \defeq
    (\Theta \setminus \{\Id\}) \big/ {\sim},
    \qquad
    \theta \sim \theta^{-1}.
    \label{eq:Theta-double-prime}
  \end{equation}
\end{definition}

The Schottky group $\Theta$ carries a rooted tree of reduced words with the identity at the root (\Cref{fig:level-tree}).
Level~$\ell$ comprises the vertices at distance~$\ell$ from the root, i.e., the reduced words of length~$\ell$.
Define
\begin{equation}
  \Theta''_\ell
  \defeq
  \{\theta \in \Theta'' : \level(\theta) = \ell\}.
  \label{eq:Theta-ell}
\end{equation}
Each vertex of the tree corresponds one-to-one with a reduced word~\eqref{eq:schottky-group}.
Following the index sequence $(j_1, \ldots, j_\ell)$ from the root produces the successive compositions $\psi_i = \theta_{j_i} \cdots \theta_{j_1}$, and the vertex reached at level~$\ell$ is precisely the reduced word $\psi_\ell$.
If $\psi = \theta_{j_\ell} \cdots \theta_{j_1} \in \Theta \setminus \{\Id\}$, its children are $\theta_k\psi$ for the $2M - 1$ indices $k$ with $k\neq -j_\ell$.
For $\ell \ge 1$, $\Card \ab\{\Theta''_\ell\} = M(2M - 1)^{\ell - 1}$, so $\Card \ab\{\Theta''_1\} = M$ and $\Card \ab\{\Theta''_2\} = M(2M - 1)$.
\subsection{The Schottky--Klein prime function}
\label{sec:skprime}

For any two points $\zeta, \alpha \in \Domain$, the \emph{Schottky--Klein prime function} is given~\cite{schottky1887Ueber,crowdy2020Solving} by the infinite product
\begin{equation}
  \omega(\zeta, \alpha)
  \defeq(\zeta - \alpha)
  \prod_{\theta \in \Theta''} R_{\zeta, \alpha}(\theta),
  \qquad
  R_{\zeta, \alpha}(\theta)
  \defeq
  \frac{\ab(\theta(\zeta) - \alpha)\ab(\theta(\alpha) - \zeta)}
  {\ab(\theta(\zeta) - \zeta)\ab(\theta(\alpha) - \alpha)},
  \label{eq:sk-prime}
\end{equation}
where the product runs over one representative of each class in $\Theta''$.
The factor $R_{\zeta, \alpha}(\theta)$ is the inverse of a cross-ratio:
\begin{equation}
  R_{\zeta, \alpha}(\theta)
  =
  \frac{1}{R\ab(\zeta, \alpha, \theta(\zeta), \theta(\alpha))},
  \qquad
  R(a, b, c, d) \defeq
  \frac{(c - a)(d - b)}{(c - b)(d - a)}.
  \label{eq:R-cross-ratio}
\end{equation}

The classical \emph{level truncation} retains all group elements up to a prescribed word length~$L_{\max}$.
To evaluate~$\omega$ numerically, we replace the infinite product~\eqref{eq:sk-prime} by the finite product over those words.
Define
\begin{equation}
  \Theta''_{\le L_{\max}}
  \defeq
  \bigcup_{\ell = 1}^{L_{\max}}\Theta''_\ell
  =
  \{\theta \in \Theta'' : \level(\theta) \le L_{\max}\},
  \label{eq:Theta-leqL}
\end{equation}
and the partial product
\begin{equation}
  \omega_{L_{\max}}(\zeta, \alpha)
  \defeq(\zeta - \alpha)
  \prod_{\theta \in \Theta''_{\le L_{\max}}}
  R_{\zeta, \alpha}(\theta).
  \label{eq:level-trunc}
\end{equation}
\Cref{fig:level-tree} illustrates the tree of the Schottky group for $M = 2$: the partition of $\Theta$ by level and the elements included in $\Theta''_{\le L_{\max}}$ for a fixed $L_{\max}$.
Because $\Card \ab\{\Theta''_\ell\} = M(2M - 1)^{\ell - 1}$, computing $\omega_{L_{\max}}$ requires exponentially many factors as $L_{\max}$ grows.

\begin{figure}[t]
  \centering
  \resizebox{0.72\textwidth}{!}{\begin{tikzpicture}[x=0.82cm, y=0.88cm, font=\small]
  \SchottkyTreeMTwoLevelPartition

  \node[level box, fit=(t1)(t2)(tm1)(tm2), label={[cud level label]above:level $1$}] {};
  \node[level box, fit=(t1t1)(t2t1)(tm2t1)(t1m2)(tm1m2)(tm2m2)(dots2)(dotsm1), label={[cud level label]above:level $2$}] {};
  \node[level box, fit=(t1m2t1)(tm1m2t1)(tm2m2t1)(t2tm1m2)(dotsT1T1)(dotsT2T1)(dotsT1M2)(dotsM2M2)(dotsTm1M2), label={[cud level label]above:level $3$}] {};
  \node[level box, fit=(tm1m2m2t1)(dotsL3a)(dotsL3b)(dotsL4)(dotsL3c), label={[cud level label]above:level $4$}] {};
\end{tikzpicture}}%

  \caption{Tree of reduced words in the Schottky group for $M = 2$ holes.
    Children of a vertex $\psi = \theta_{j_\ell} \cdots \theta_{j_1}$ are $\theta_k\psi$ with $k \neq -j_\ell$ (left composition).
    Level~$\ell$ comprises all reduced words of length~$\ell$; level truncation at~$L_{\max}$ retains every element in $\Theta''_{\le L_{\max}} = \bigcup_{\ell = 1}^{L_{\max}}\Theta''_\ell$.}
  \label{fig:level-tree}
\end{figure}
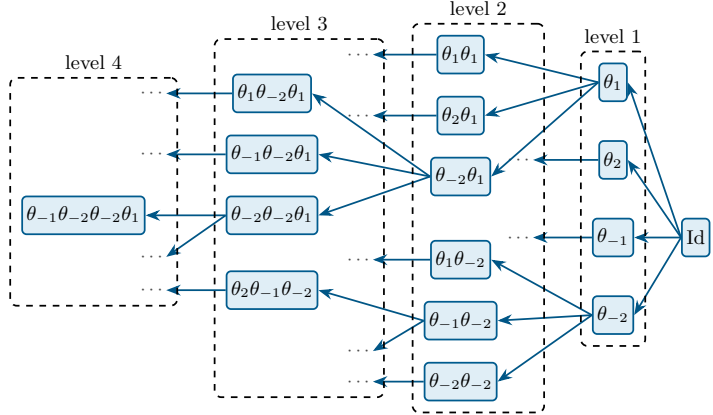

\subsection{Pollicott's counting theorem}
\label{sec:pollicott}

We quote the following counting theorem of Pollicott~\cite{pollicott2021Schottky}.
The \emph{Poincar\'e upper half-space} is
\begin{equation}
  \UpperHalf^3
  = \{(z, t) \in \Complex \times \Reals : t > 0\},
  \qquad
  \dl{s}^2 = \frac{ \dl{x}^2 + \dl{y}^2 + \dl{t}^2}{t^2},
  \label{eq:H3}
\end{equation}
where $z = x + \iunit y$.
Each M\"obius transformation on~$\Complex$ extends to an orientation-preserving isometry of $(\UpperHalf^3, \dl{s}^2)$; in particular, every $\theta \in \Theta$ acts on~$\UpperHalf^3$.
A \emph{geodesic} in $\UpperHalf^3$ with endpoints $u, u' \in \Complex \cup \{\infty\}$ is denoted $[u, u']$.
For $\theta \in \Theta$ and such endpoints, $\theta$ sends $[u, u']$ to the geodesic with endpoints $\theta(u)$ and $\theta(u')$; we write
\begin{equation}
  \theta(\gamma) \defeq [\theta(u), \theta(u')],
  \qquad \gamma=[u, u'].
  \label{eq:theta-on-geodesic}
\end{equation}
The \emph{hyperbolic distance} between two geodesics is written $d(\gamma_1, \gamma_2)$.
For fixed $\zeta, \alpha \in \Domain$, set $\gamma \defeq[\zeta, \alpha]$ and define, for $\theta \in \Theta''$,
\begin{equation}
  d(\theta) \defeq d\ab(\gamma, \theta(\gamma)).
  \label{eq:d-theta}
\end{equation}
Further background on the hyperbolic distance in $\UpperHalf^3$ is given in~\cite{pollicott2021Schottky,marden2016Hyperbolic,fenchel1989Elementary}.

\begin{theorem}[Pollicott~\cite{pollicott2021Schottky}]
  \label{thm:pollicott}
  For fixed $\zeta, \alpha \in \Domain$, there exist a constant $C > 0$ and an exponent $\delta \in (0, 2)$, depending only on~$\Theta$, such that
  \begin{equation}
    \Card \ab\{\theta \in \Theta :
    \abs{R_{\zeta, \alpha}(\theta) - 1} \ge 1/T\}
    \sim CT^{\delta}
    \qquad \text{as }T \to \infty.
    \label{eq:pollicott-count}
  \end{equation}
  The exponent $\delta$ coincides with the Hausdorff dimension of the limit set of $\Theta$.
\end{theorem}

We use \cref{thm:pollicott} only as a quoted fact; the same statement is Theorem~1.3 in~\cite{pollicott2021Schottky}.
This statement is invoked in \cref{sec:counting} to count the potential level sets that underlie the relative-error bound of \cref{thm:error} and the complexity estimate of \cref{prop:alg}.
The term-wise estimate $\abs{R_{\zeta, \alpha}(\theta) - 1} \sim 4 e^{-d(\theta)}$ relating factor size to hyperbolic distance is also taken from~\cite{pollicott2021Schottky}.

% ============================================================
\section{Potential and proposed truncation}
\label{sec:potential}

Unlike level truncation, which orders group elements by word length, the scheme described in this section ranks each factor $R_{\zeta, \alpha}(\theta)$ in~\eqref{eq:sk-prime} according to a scalar \emph{potential} $U_{\zeta, \alpha}(\theta)$ defined directly from cross-ratios and motivated by the hyperbolic picture of \cref{sec:pollicott}.
\Cref{sec:def-potential} introduces $U$, records the exact relation $\abs{R_{\zeta, \alpha}(\theta) - 1} = e^{-U_{\zeta, \alpha}(\theta)}$, and explains why large values of $U_{\zeta, \alpha}(\theta)$ correspond to negligible factors in the S--K prime product.
\Cref{sec:defn-trunc} defines potential truncation at a cutoff $U_{\max}$ and contrasts it with level truncation on the tree of the Schottky group.

\subsection{Potential associated with a group element}
\label{sec:def-potential}

\begin{definition}[Potential]
  \label{def:potential}
  For $\theta \in \Theta \setminus \{\Id\}$, define the \emph{potential} as
  \begin{equation}
    U_{\zeta, \alpha}(\theta)
    \defeq - \log \abs{R_{\zeta, \alpha}(\theta) - 1}.
    \label{eq:potential-def}
  \end{equation}
\end{definition}

By~\eqref{eq:R-cross-ratio}, $R_{\zeta, \alpha}(\theta)$ is the inverse of the cross-ratio of the four points $\zeta$, $\alpha$, $\theta(\zeta)$, and $\theta(\alpha)$; thus, for each $\theta \in \Theta$, $U_{\zeta, \alpha}(\theta)$ is a real number determined by the reference pair $(\zeta, \alpha)$ and the hole data $\{\delta_m, q_m\}_{m = 1}^{M}$.
This definition is motivated by Pollicott's estimate~\cite{pollicott2021Schottky} $\abs{R_{\zeta, \alpha}(\theta) - 1} \sim 4 e^{-d(\theta)}$ as $d(\theta) \to \infty$, where $d(\theta)$ is the hyperbolic distance between the geodesics $[\zeta, \alpha]$ and $\theta([\zeta, \alpha])$ in $\UpperHalf^3$, as defined in \cref{sec:pollicott}.
From~\eqref{eq:potential-def}, the exact factor-wise relation is
\begin{equation}
  \abs{R_{\zeta, \alpha}(\theta) - 1} = e^{-U_{\zeta, \alpha}(\theta)}.
  \label{eq:term-as-potential}
\end{equation}

To see why $U$ governs the truncation of~\eqref{eq:sk-prime}, write each factor as a perturbation of unity:
\begin{equation}
  \omega(\zeta, \alpha)
  = (\zeta - \alpha)\prod_{\theta \in \Theta''} R_{\zeta, \alpha}(\theta)
  = (\zeta - \alpha)\prod_{\theta \in \Theta''}\ab(1 + (R_{\zeta, \alpha}(\theta) - 1)).
  \label{eq:omega-as-perturbation}
\end{equation}
By~\eqref{eq:term-as-potential}, the perturbation in the $\theta$-factor has magnitude $e^{-U_{\zeta, \alpha}(\theta)}$.
Therefore, elements with large $U_{\zeta, \alpha}(\theta)$ make the smallest multiplicative corrections and are natural candidates for omission in a partial product.
This suggests the truncation $U_{\zeta, \alpha}(\theta) \le U_{\max}$ introduced in~\cref{sec:defn-trunc}.

\subsection{Truncation by potential}
\label{sec:defn-trunc}

Motivated by this observation, we introduce a truncation scheme that selects group elements according to their potential rather than their word length.

\begin{definition}[Potential truncation]
  \label{def:trunc}
  For a potential cutoff $U_{\max} > 0$, define
  \begin{equation}
    \Theta''_{\le U_{\max}}
    \defeq
    \{\theta \in \Theta'' : U_{\zeta, \alpha}(\theta) \le U_{\max}\},
    \label{eq:Theta-leqU}
  \end{equation}
  and the partial product
  \begin{equation}
    \omega_{U_{\max}}(\zeta, \alpha)
    \defeq
    (\zeta - \alpha)
    \prod_{\theta \in \Theta''_{\le U_{\max}}}
    R_{\zeta, \alpha}(\theta).
    \label{eq:trunc-def}
  \end{equation}
\end{definition}

Unlike level truncation~\eqref{eq:level-trunc}, whether a given $\theta \in \Theta''$ belongs to $\Theta''_{\le U_{\max}}$ is determined by $U_{\zeta,\alpha}(\theta)$, and hence by $(\zeta,\alpha)$ and the hole configuration, rather than by word length alone.
Equivalently, $\Theta''_{\le U_{\max}}$ comprises those $\theta \in \Theta''$ with $\abs{R_{\zeta, \alpha}(\theta) - 1} \ge e^{-U_{\max}}$.
On the tree of the Schottky group (\Cref{fig:level-tree}), level truncation fills the slab of all vertices up to depth~$L_{\max}$ (\Cref{fig:level-trunc-tree}), whereas potential truncation keeps only the vertices with $U_{\zeta, \alpha}(\theta) \le U_{\max}$ (\Cref{fig:potential-trunc-tree}), a sparse selection that typically cuts across levels.
Intuitively, the latter is more economical than retaining an entire level slab when many shallow words contribute little to the product.

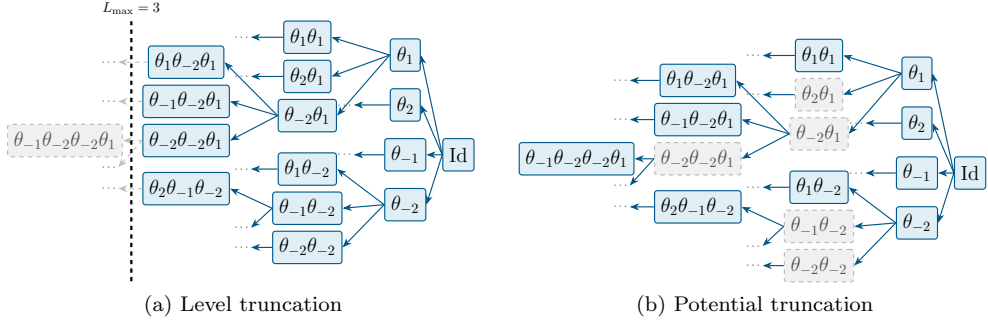
\begin{figure}[t]
  \centering
  \subfloat[Level truncation]{\label{fig:level-trunc-tree}%
  \resizebox{0.48\textwidth}{!}{\begin{tikzpicture}[x=0.82cm, y=0.88cm, font=\Large]
  \tikzset{
    schottky base/.append style={
      font=\Large,
      inner sep=4.5pt,
      minimum height=8mm,
    },
  }
  \SchottkyTreeLmaxThree

  \draw[frontier line] (\SchottkyCoordFrontierLmaxThree, -3.55) -- (\SchottkyCoordFrontierLmaxThree, 3.75);
  \node[anchor=south, cud level label, font=\normalsize] at (\SchottkyCoordFrontierLmaxThree, 3.80) {$L_{\max} = 3$};
\end{tikzpicture}}%
}
  \hfill
  \subfloat[Potential truncation]{\label{fig:potential-trunc-tree}%
  \resizebox{0.48\textwidth}{!}{% Potential truncation schematic on the Schottky-group tree ($M = 2$).

\begin{tikzpicture}[x=0.82cm, y=0.88cm, font=\Large]
  \tikzset{
    schottky base/.append style={
      font=\Large,
      inner sep=4.5pt,
      minimum height=8mm,
    },
  }
  \SchottkyPlacePotentialExample
  \SchottkyEdgesPotentialExample
\end{tikzpicture}}%
}
  \caption{Truncation on the tree of the Schottky group for $M = 2$.
    (a)~Level truncation fills the slab of all vertices up to depth~$L_{\max}$ (\tikzlegendkept: $\level(\theta) \le L_{\max}$; \tikzlegendpruned: $\level(\theta) > L_{\max}$).
    (b)~Potential truncation retains only vertices with $U_{\zeta, \alpha}(\theta) \le U_{\max}$ (\tikzlegendkept: $U_{\zeta, \alpha}(\theta) \le U_{\max}$; \tikzlegendpruned: $U_{\zeta, \alpha}(\theta) > U_{\max}$).
    The selection is scattered across levels rather than a full level slab.}
  \label{fig:truncation-trees}
\end{figure}

Words of any level can be included or excluded, depending on the domain data.
Therefore, the two practical issues are as follows: (i)~to enumerate $\Theta''_{\le U_{\max}}$ without exploring every vertex up to large depth in the tree of the Schottky group, as level truncation does in \Cref{fig:level-trunc-tree}, and (ii)~to quantify the truncation error $\abs{\omega_{U_{\max}} - \omega}/\abs{\omega}$ and the cost of the enumeration as functions of~$U_{\max}$.
% ============================================================
\section{Efficient enumeration algorithm}
\label{sec:algorithm}

Throughout this section, we fix reference points $\zeta, \alpha \in \Domain$.
All quantities introduced from this point onward are evaluated for this fixed pair, together with the hole data and the relevant group element or word where appropriate.
For readability, we therefore suppress the fixed parameter pair and write, for example, $R(\theta)$ and $U(\theta)$ (or $U(\psi)$), restoring explicit dependence on $(\zeta, \alpha)$ only when needed.

We now present an algorithm that enumerates the truncation set $\Theta''_{\le U_{\max}}$ in~\eqref{eq:Theta-leqU} and assembles the partial product $\omega_{U_{\max}}$ in~\eqref{eq:trunc-def}.

A direct scan of the tree of the Schottky group cannot terminate on its own: the tree is infinite and $U_{\max}$ does not bound the word length.
Therefore, we use a depth-first search that, on reaching a vertex $\psi$, decides whether the subtree at $\psi$ may contain elements of $\Theta''_{\le U_{\max}}$ and, if not, discards it.
The decision requires a lower bound on $U$ that holds for every descendant of $\psi$.
\Cref{thm:increment}, proved in \cref{sec:increment}, supplies one-step lower bounds on the change of~$U$ when a generator is prepended; those bounds depend only on the consecutive generators in the reduced word.
\Cref{sec:graph-G} turns them into a path bound on the tree and a walk-weight bound on a finite directed graph~$G$.
\Cref{sec:wmin} introduces the minimum walk weight $W_{\min}$ of~$G$, which is the worst-case slack used by the pruning rule.
\Cref{sec:enum-alg} states the depth-first traversal that produces $\Theta''_{\le U_{\max}}$ and $\omega_{U_{\max}}$; the Bellman--Ford algorithm that computes $W_{\min}$ appears in the supplementary material (\cref{sup:alg:bellman-ford}).

\subsection{From the potential increment to a finite weighted graph}
\label{sec:graph-G}

Write $\psi = \theta_{j_\ell} \cdots \theta_{j_1}$ for a reduced word and let $k \neq -j_\ell$.
The one-step change of potential under the prepend $\psi \mapsto \theta_k\psi$ is the \emph{potential increment}
\begin{equation}
  \Delta_k(\psi)
  \defeq
  U(\theta_k\psi) - U(\psi).
  \label{eq:increment-def-alg}
\end{equation}
\Cref{thm:increment} gives constants $\Delta_{j_\ell, k}^{\lb}$ and $\Delta_{j_\ell, k}^{\ub}$, depending only on the consecutive generators $\theta_{j_\ell}, \theta_k$ and on $(\zeta, \alpha)$, such that
\begin{equation}
  \Delta_{j_\ell, k}^{\lb}
  \le
  \Delta_k(\psi)
  \le
  \Delta_{j_\ell, k}^{\ub}.
  \label{eq:increment-bound-alg}
\end{equation}
Here, the superscripts ``lb'' and ``ub'' stand for lower bound and upper bound.
Only the lower bounds enter the construction below: in particular,
\begin{equation}
  U(\theta_k\psi) - U(\psi) \ge \Delta_{j_\ell, k}^{\lb},
  \label{eq:increment-lower}
\end{equation}
and the $2M(2M - 1)$ quantities $\Delta_{j, k}^{\lb}$ are computed once from~\eqref{eq:deltajk-lb} in $O(M^2)$ time before the search starts.

Now, fix a vertex $\psi_\ell = \theta_{j_\ell} \cdots \theta_{j_1}$ at level~$\ell$ and consider a descendant $\psi_{\ell'}$ at some level $\ell' \ge \ell$.
The path from $\psi_\ell$ to $\psi_{\ell'}$ in the tree of the Schottky group (\Cref{fig:tree-to-graph}) is the chain $\psi_\ell, \psi_{\ell + 1}, \ldots, \psi_{\ell'}$ with $\psi_m = \theta_{j_m} \cdots \theta_{j_1}$.
Applying~\eqref{eq:increment-lower} at each prepend gives
\begin{equation}
  U(\psi_{\ell'}) \ge U(\psi_\ell) + \sum_{m = \ell + 1}^{\ell'} \Delta_{j_{m - 1}, j_m}^{\lb}.
  \label{eq:prefix-lower}
\end{equation}
The sum on the right depends only on the index list $(j_\ell, j_{\ell + 1}, \ldots, j_{\ell'})$, not on the contents of~$\psi_\ell$.

This makes it natural to work on a finite directed graph rather than on the infinite tree.
Let $G = (V, E, w)$ be the graph with vertices
\begin{equation*}
  V = \{\theta_{-M}, \ldots, \theta_{-1}, \theta_1, \ldots, \theta_M\},
\end{equation*}
edges $\theta_j \to \theta_k$ for every pair with $k \neq -j$, and weights $w(\theta_j, \theta_k) = \Delta_{j, k}^{\lb}$.
The constraint $k\neq -j$ excludes only the immediate cancellation $\theta_{-j}\theta_j = \Id$.
In particular, the self-loops $\theta_j \to \theta_j$ are present.
The graph has $2M$ vertices and $2M(2M - 1)$ edges.
For $M = 2$, there are four vertices and twelve edges (\Cref{fig:tree-to-graph}, right).

\begin{figure}[t]
  \centering
  \resizebox{0.88\textwidth}{!}{% Paper layout: Schottky tree (left) and graph $G$ (right).
% Presentation uses schottky-tree-path-to-walk.tex (stacked vertically).
% Modest panel scales only (not extreme): tree slightly smaller, $G$ larger.
\begin{tikzpicture}[x=0.82cm, y=0.88cm, font=\Large]
  \tikzset{
    schottky base/.append style={
        font=\Large,
        inner sep=4.5pt,
        minimum height=8mm,
      },
    graph node fixed/.style={
        graph node,
        font=\Large,
        minimum size=11mm,
        inner sep=0pt,
        align=center,
      },
    % Compensate panel scales so tree/graph $\Delta$ labels read at similar size.
    delta label tree/.style={
        fill=white,
        fill opacity=0.94,
        inner sep=0.5pt,
        font=\normalsize,
        text=cudVermillionDraw,
      },
    delta label graph/.style={
        fill=white,
        fill opacity=0.94,
        inner sep=1pt,
        font=\footnotesize,
        text=cudVermillionDraw,
      },
  }

  \def\TreePanelScale{0.82}
  \def\GraphPanelScale{1.28}
  \def\TreeCenterY{0.35}
  \def\GraphCenterY{0.975}
  \def\GraphShiftX{11.0}

  \begin{scope}[scale=\TreePanelScale, transform shape]
    \SchottkyPlaceId
    \node[schottky path] (t1) at (\SchottkyCoordLone, 2.80) {$\theta_1$};
    \node[schottky active] (t2) at (\SchottkyCoordLone, 1.40) {$\theta_2$};
    \node[schottky active] (tm1) at (\SchottkyCoordLone, 0.00) {$\theta_{-1}$};
    \node[schottky active] (tm2) at (\SchottkyCoordLone, -1.40) {$\theta_{-2}$};
    \SchottkyPlaceEllipsis
    \SchottkyPlaceLTwoFromTone
    \SchottkyPlaceLTwoFromTmtwo
    \SchottkyPlaceLThreeFromTmTwoTone
    \SchottkyPlaceLThreeDots
    \node[schottky path] (tm2t1) at (\SchottkyCoordLtwo, 1.10) {$\theta_{-2}\theta_1$};
    \node[schottky path] (tm1m2t1) at (\SchottkyCoordLthree, 1.50) {$\theta_{-1}\theta_{-2}\theta_1$};

    \stree{id}{t1}
    \stree{id}{t2}
    \stree{id}{tm1}
    \stree{id}{tm2}
    \SchottkyEdgesEllipsis
    \stree{t1}{t1t1}
    \stree{t1}{t2t1}
    \path (t1.west) edge[schottky edge path]
    node[delta label tree, above, pos=0.50, yshift=-3mm] {$\Delta_{1,-2}^{\lb}$} (tm2t1.east);
    \stree{tm2}{t1m2}
    \stree{tm2}{tm1m2}
    \stree{tm2}{tm2m2}
    \stree{tm2t1}{tm2m2t1}
    \path (tm2t1.west) -- node[delta label tree, above, pos=0.50, yshift=-5mm] {$\Delta_{-2,-1}^{\lb}$} (tm1m2t1.east);
    \stree{tm2t1}{t1m2t1}
    \path (tm2t1.west) edge[schottky edge path] (tm1m2t1.east);
    \SchottkyEdgesLThreeDots

    \node[anchor=north, font=\Large] at (4.0, -3.75) {Schottky tree};
  \end{scope}

  \pgfmathsetmacro{\GraphShiftY}{\TreePanelScale * \TreeCenterY - \GraphPanelScale * \GraphCenterY}

  \begin{scope}[scale=\GraphPanelScale, transform shape, shift={(\GraphShiftX, \GraphShiftY)}]
    \node[graph node fixed] (g1) at (-1.25, 0) {$\theta_1$};
    \node[graph node fixed] (g2) at (1.25, 0) {$\theta_2$};
    \node[graph node fixed] (gm1) at (1.25, 1.95) {$\theta_{-1}$};
    \node[graph node fixed] (gm2) at (-1.25, 1.95) {$\theta_{-2}$};

    \foreach \a/\b in {g1/g2, g2/gm1} {
        \draw[graph edge, bend left=18] (\a) to (\b);
        \draw[graph edge, bend left=18] (\b) to (\a);
      }
    \draw[graph edge, bend left=18] (gm1) to (gm2);
    \draw[graph edge, bend left=18] (gm2) to (g1);
    \draw[graph edge] (g1) to[loop left, looseness=5, min distance=4mm] (g1);
    \draw[graph edge] (g2) to[loop right, looseness=5, min distance=4mm] (g2);
    \draw[graph edge] (gm1) to[loop right, looseness=5, min distance=4mm] (gm1);
    \draw[graph edge] (gm2) to[loop left, looseness=5, min distance=4mm] (gm2);

    \draw[graph edge path, bend left=18] (g1) to
    node[delta label graph, left, pos=0.55, xshift=-1pt] {$\Delta_{1,-2}^{\lb}$} (gm2);
    \draw[graph edge path, bend left=18] (gm2) to
    node[delta label graph, above, pos=0.50] {$\Delta_{-2,-1}^{\lb}$} (gm1);

    \node[anchor=north, font=\normalsize] at (0, -0.75) {graph $G$};
  \end{scope}
\end{tikzpicture}}%

  \caption{Tree of the Schottky group (left) and the finite weighted graph~$G$ (right) for $M = 2$.
    A path in the tree, e.g., $\theta_1 \to \theta_{-2}\theta_1 \to \theta_{-1}\theta_{-2}\theta_1$, corresponds to a walk in~$G$ that traces the same index sequence and accumulates the edge weights $\Delta_{j, k}^{\lb}$.}
  \label{fig:tree-to-graph}
\end{figure}
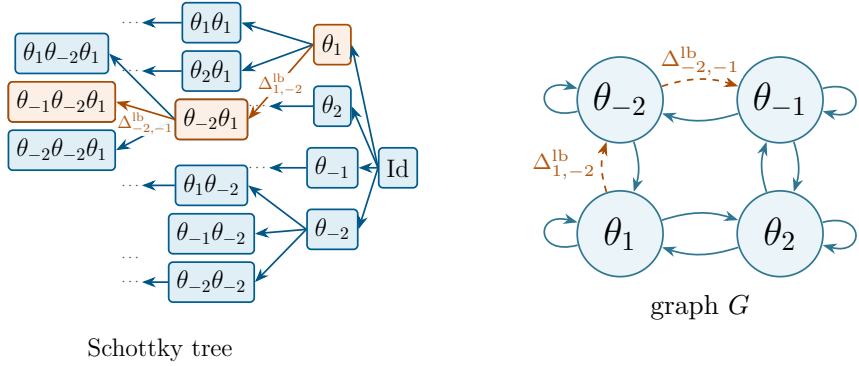

Any index sequence $(i_1, \ldots, i_n)$ with $i_{m + 1}\neq -i_m$ traces a \emph{walk} $\theta_{i_1} \to \cdots \to \theta_{i_n}$ in~$G$.
We write the sum of the corresponding lower-bound edge weights as
\begin{equation}
  W(i_1, \ldots, i_n) \defeq \sum_{m = 2}^{n} \Delta_{i_{m - 1}, i_m}^{\lb}.
  \label{eq:walk-weight}
\end{equation}
With this notation,~\eqref{eq:prefix-lower} reads
\begin{equation}
  U(\psi_{\ell'}) \ge U(\psi_\ell) + W(j_\ell, j_{\ell + 1}, \ldots, j_{\ell'});
  \label{eq:descendant-walk-weight}
\end{equation}
the walk on the right is that traced by the index list of the path from $\psi_\ell$ to $\psi_{\ell'}$ (\Cref{fig:tree-to-graph}).

\subsection{Minimum walk weight and the pruning rule}
\label{sec:wmin}

To prune the subtree at $\psi_\ell$, we require a lower bound on $U(\psi_{\ell'})$ that holds uniformly over every descendant $\psi_{\ell'}$.
By~\eqref{eq:descendant-walk-weight}, this reduces to bounding the walk weight $W(j_\ell, \ldots, j_{\ell'})$ from below over every choice of extension $j_{\ell + 1}, \ldots, j_{\ell'}$.
Each such extension is a walk in $G$ starting at $\theta_{j_\ell}$.
Thus, it suffices to bound the weight of any walk in $G$ by a single constant.
Therefore, we define the \emph{minimum walk weight} of $G$ as
\begin{equation}
  W_{\min} \defeq \min_{\pi} \sum_{e \in \pi} w(e),
  \label{eq:wmin-def}
\end{equation}
where the minimum is over all walks $\pi$ in $G$, including the empty walk.
The empty walk has weight zero, so $W_{\min} \le 0$.
By construction, $W(j_\ell, \ldots, j_{\ell'}) \ge W_{\min}$ for every extension, hence
\begin{equation}
  U(\psi_{\ell'}) \ge U(\psi_\ell) + W_{\min}
  \qquad \text{for every descendant } \psi_{\ell'} \text{ of } \psi_\ell.
  \label{eq:descendant-lower}
\end{equation}
The pruning rule follows directly: if $U(\psi_\ell) > U_{\max} - W_{\min}$, then $U(\psi_{\ell'}) > U_{\max}$ for every descendant, and the subtree at $\psi_\ell$ can be discarded.

If $G$ has no negative cycle, $W_{\min} \in (-\infty, 0]$ is finite and is computed by the Bellman--Ford algorithm of graph theory (\cref{sup:alg:bellman-ford}).
When every edge weight is nonnegative, $W_{\min} = 0$ and the pruning rule reduces to $U(\psi_\ell) > U_{\max}$; the same algorithm still applies.

A negative cycle in~$G$ makes $W_{\min}=-\infty$ and the pruning rule vacuous, so \cref{alg:enumerate} aborts when such a cycle is detected by the Bellman--Ford algorithm.
\Cref{rem:upper-bound} explains why a negative cycle blocks pruning, and distinguishes this obstruction from genuine divergence of the product.

\subsection{The algorithm}
\label{sec:enum-alg}

The enumeration proceeds in three stages: precompute the edge weights $\Delta_{j, k}^{\lb}$ in closed form as described in \cref{sec:edge-weights-closed}, precompute $W_{\min}$ on~$G$ according to \cref{sup:alg:bellman-ford}, and traverse the tree of the Schottky group depth-first while applying the rule of~\cref{sec:wmin} and assembling the partial product~\eqref{eq:trunc-def} using \cref{alg:enumerate}.
\begin{algorithm}[t]
  \caption{Enumerate $\Theta''_{\le U_{\max}}$ and assemble $\omega_{U_{\max}}$.}
  \label{alg:enumerate}
  \begin{algorithmic}[1]
    \STATE \textbf{Input:} reference points $\zeta, \alpha \in \Domain$; cutoff $U_{\max} > 0$; lower bounds $\{\Delta_{j, k}^{\lb}\}$.
    \STATE \textbf{Output:} the list $\mathcal{O} = \Theta''_{\le U_{\max}}$ and the partial product $\omega_{U_{\max}}$.
    \STATE Run \cref{sup:alg:bellman-ford} on $G$ to obtain $W_{\min} \le 0$; abort if a negative cycle is detected.
    \STATE Initialize stack $S\Leftarrow [ \Id ]$, output list $\mathcal{O}\Leftarrow\emptyset$, partial product $\omega\Leftarrow\zeta - \alpha$.
    \WHILE{$S$ is non-empty}
    \STATE Pop $\theta$ from $S$.
    \STATE $K\Leftarrow\emptyset$.
    \IF{$\theta = \Id$}
    \STATE $K\Leftarrow\{\pm 1, \ldots, \pm M\}$.
    \ELSE
    \STATE Compute $U(\theta)$ via~\eqref{eq:potential-def}.
    \IF{$U(\theta) \le U_{\max} - W_{\min}$}
    \IF{$U(\theta) \le U_{\max}$ and $\theta \in \Theta''$}
    \STATE Append $\theta$ to $\mathcal{O}$ and update $\omega\Leftarrow \omega\cdot R(\theta)$.
    \ENDIF
    \STATE For $\theta = \theta_{j_\ell} \cdots \theta_{j_1}$, set $K\Leftarrow\{k \in \{\pm 1, \ldots, \pm M\} : k\neq -j_\ell\}$.
    \ENDIF
    \ENDIF
    \FORALL{$k \in K$}
    \STATE Push $\theta_k\theta$ onto $S$.
    \ENDFOR
    \ENDWHILE
    \STATE \textbf{return} $\mathcal{O}$ and $\omega_{U_{\max}}\Leftarrow \omega$.
  \end{algorithmic}
\end{algorithm}

The exact potential $U(\theta)$ is only evaluated via~\eqref{eq:potential-def} at words that survive the pruning rule.
Because $\Theta''_{\le U_{\max}}\subset\Theta''$ omits inverses (see~\eqref{eq:Theta-double-prime}), the implementation appends to $\mathcal{O}$ only one of $\theta$ and $\theta^{-1}$ whenever both appear in the tree.
We keep the representative whose index sequence is lexicographically smaller.
Correctness and complexity are formalized in \cref{prop:alg} of \cref{sec:complexity}, which rests on the counting estimate of \cref{lem:counting} in \cref{sec:counting}.

\begin{remark}[Negative cycles in~$G$]
  \label{rem:upper-bound}
  Suppose $G$ has a negative cycle when the edge weights are the lower bounds $\Delta_{j, k}^{\lb}$.
  Looping around that cycle drives the walk-weight lower bound arbitrarily far in the negative direction, and so~\eqref{eq:descendant-lower} supplies no finite floor on the potential of any descendant.
  In other words, no matter how deep one descends in the tree of the Schottky group, the lower bound need not force $U$ to become large.
  Deep vertices may therefore still contribute substantially to the infinite product.
  Thus, the search cannot safely prune a subtree midway, and \cref{alg:enumerate} aborts.
  This does not by itself certify divergence of the full product~\eqref{eq:sk-prime}.
  If, however, the analogous graph built with the upper bounds $\Delta_{j, k}^{\ub}$ admits a negative cycle, then $U$ can be forced to decrease along arbitrarily long words, and the truncated product does not converge.
  In contrast, if the lower-bound graph has no negative cycle for any $(\zeta, \alpha) \in \Domain \times \Domain$, then the infinite product~\eqref{eq:sk-prime} converges for every such pair.
  This is a sufficient condition: it does not settle the convergence of~\eqref{eq:sk-prime} in full generality.
\end{remark}
% ============================================================
\section{A uniform bound on the potential increment}
\label{sec:increment}

The edge weights $\Delta_{j,k}^{\lb}$ in~\eqref{eq:walk-weight} are derived from a single bound on the potential increment that is independent of the tail of the word, and the pruning step of \cref{alg:enumerate} follows from the resulting lower bound.
\Cref{sec:increment-why} explains why such a uniform bound exists at all: a short cross-ratio computation, together with the geometry of reduced words, simplifies the question to a one-dimensional optimization on a boundary circle.
\Cref{sec:increment-statement} states the resulting bound and its closed-form constants.
\Cref{sec:edge-weights-closed} shows that those constants can be evaluated in closed form, and therefore need only be precomputed once.

\subsection{Existence of a uniform bound}
\label{sec:increment-why}

Fix a reduced word $\psi = \theta_{j_\ell} \cdots \theta_{j_1} \in \Theta \setminus \{\Id\}$ and an index $k\neq -j_\ell$.
Restoring the fixed pair $(\zeta, \alpha)$, the potential increment~\eqref{eq:increment-def-alg} is
\begin{equation}
  \Delta_k(\psi; \zeta, \alpha)
  \defeq
  U_{\zeta, \alpha}(\theta_k\psi) - U_{\zeta, \alpha}(\psi),
  \qquad j = j_\ell,
  \label{eq:increment-def}
\end{equation}
and $\theta_k\psi$ is the child of $\psi$ in the tree of the Schottky group.
The next four steps explain why $\Delta_k(\psi; \zeta, \alpha)$ admits upper and lower bounds that depend only on $\theta_j$, $\theta_k$, $\zeta$, and $\alpha$.

\paragraph*{Step~1: rewrite the increment using two image points}
Set
\begin{equation*}
  \zeta' \defeq \psi(\zeta), \qquad
  \alpha' \defeq \psi(\alpha).
\end{equation*}
By~\eqref{eq:R-cross-ratio},
$R_{\zeta, \alpha}(\theta_k\psi) = R(\zeta, \alpha, \theta_k(\zeta'), \theta_k(\alpha'))^{-1}$ and $R_{\zeta, \alpha}(\psi) = R(\zeta, \alpha, \zeta', \alpha')^{-1}$.
By the M\"obius invariance of the cross-ratio,
\begin{equation*}
  R\ab(\zeta, \alpha, \theta_k(\zeta'), \theta_k(\alpha'))
  = R\ab(\theta_k^{-1}(\zeta), \theta_k^{-1}(\alpha), \zeta', \alpha').
\end{equation*}
From~\eqref{eq:potential-def} and~\eqref{eq:increment-def},
\begin{equation}
  \Delta_k(\psi; \zeta, \alpha)
  =
  \log \abs{
    \frac{R\ab(\zeta, \alpha, \zeta', \alpha')^{-1} - 1}
    {R\ab(\theta_k^{-1}(\zeta), \theta_k^{-1}(\alpha), \zeta', \alpha')^{-1} - 1}
  }.
  \label{eq:increment-rewrite}
\end{equation}
Once $\zeta$, $\alpha$, and $\theta_k$ have been fixed, the right-hand side depends only on $(\zeta', \alpha')$.

\paragraph*{Step~2: confine $(\zeta', \alpha')$ to $\closure{D_j}$}
Write $\psi = \theta_j\tau$ with $\tau \defeq \theta_{j_{\ell - 1}} \cdots \theta_{j_1}$.
Because $\psi$ is a reduced word and $j = j_\ell$ is its leftmost index, the M\"obius transformation $\theta_j$ maps $\Complex \setminus \closure{D_{-j}}$ onto $D_j$.
In particular,
\begin{equation}
  \psi(\zeta), \ \psi(\alpha) \in \closure{D_j}
  \quad\text{for every reduced $\psi$ with $j_\ell = j$, }
  \label{eq:image-in-Dj}
\end{equation}
regardless of the tail~$\tau$, and so $\zeta', \alpha' \in \closure{D_j}$.
Hence, bounding $\Delta_k(\psi; \zeta, \alpha)$ over all $\psi$ with $j_\ell = j$ reduces to bounding the right-hand side of~\eqref{eq:increment-rewrite} over all $(\zeta', \alpha') \in \closure{D_j} \times \closure{D_j}$:
\begin{align*}
  \Delta_{j, k}^{\lb}(\zeta, \alpha)
   & \defeq
  \min_{(\zeta', \alpha') \in \closure{D_j} \times \closure{D_j}}
  \log \abs{
    \frac{R\ab(\zeta, \alpha, \zeta', \alpha')^{-1} - 1}
    {R\ab(\theta_k^{-1}(\zeta), \theta_k^{-1}(\alpha), \zeta', \alpha')^{-1} - 1}
  },
  \\
  \Delta_{j, k}^{\ub}(\zeta, \alpha)
   & \defeq
  \max_{(\zeta', \alpha') \in \closure{D_j} \times \closure{D_j}}
  \log \abs{
    \frac{R\ab(\zeta, \alpha, \zeta', \alpha')^{-1} - 1}
    {R\ab(\theta_k^{-1}(\zeta), \theta_k^{-1}(\alpha), \zeta', \alpha')^{-1} - 1}
  },
\end{align*}
and
\begin{equation*}
  \Delta_{j, k}^{\lb}(\zeta, \alpha)
  \le
  \inf_{\tau}
  \Delta_k(\theta_j\tau; \zeta, \alpha)
  \le
  \Delta_k(\psi; \zeta, \alpha)
  \le
  \sup_{\tau}
  \Delta_k(\theta_j\tau; \zeta, \alpha)
  \le
  \Delta_{j, k}^{\ub}(\zeta, \alpha).
\end{equation*}
Here, ``lb'' and ``ub'' denote lower and upper bounds from the min/max operation over $(\zeta', \alpha')$; the infimum and supremum in the chain are over $\tau$, not over~$\psi$.
\paragraph*{Step~3: separate the two image points by choosing a reference at a fixed point}
Pick either of the two fixed points $z_k^{\fix}$ of $\theta_k$ in $\Complex \cup \{\infty\}$, and define, for $w \in \Complex$,
\begin{equation}
  R_k(w; \zeta) \defeq
  R\ab(w, z_k^{\fix}, \theta_k^{-1}(\zeta), \zeta),
  \qquad
  R_k(w; \alpha) \defeq
  R\ab(w, z_k^{\fix}, \theta_k^{-1}(\alpha), \alpha).
  \label{eq:Rz-def}
\end{equation}
As shown in \cref{sup:sec:cross-ratio-proofs}, the right-hand side of~\eqref{eq:increment-rewrite} is equal to
\begin{equation}
  \Delta_k(\psi; \zeta, \alpha)
  =
  - \log \abs{\theta_k'\ab(z_k^{\fix})}
  + \log \abs{R_k(\zeta'; \zeta)}
  + \log \abs{R_k(\alpha'; \alpha)},
  \label{eq:increment-separated}
\end{equation}
in which the first term depends only on $(\theta_j, \theta_k)$ and the only $\psi$-dependence enters through $w = \zeta'$ and $w = \alpha'$ in~\eqref{eq:Rz-def}.

\paragraph*{Step~4: collapse the optimization to the boundary $C_j$}
For fixed $\zeta$, $R_k(w; \zeta)$ is a M\"obius transformation in~$w$ with neither zeros nor poles in $D_j$, and so $\log \abs{R_k(w; \zeta)} = \Real{\log R_k(w; \zeta)}$ is harmonic on $D_j$.
By the maximum and minimum principles, its extrema over $\closure{D_j}$ are attained on the boundary~$C_j$.
The same holds for $\log \abs{R_k(w; \alpha)}$.
Combining this with~\eqref{eq:image-in-Dj} and~\eqref{eq:increment-separated} reduces the four-dimensional optimization over $(\zeta', \alpha') \in \closure{D_j} \times \closure{D_j}$ to two one-dimensional optimizations over $C_j$, the optima of which depend only on $\theta_j$, $\theta_k$, $\zeta$, and $\alpha$.

\subsection{Statement of the bound}
\label{sec:increment-statement}

Combining the four steps of \cref{sec:increment-why} yields the following bound.

\begin{theorem}[Uniform increment bound]
  \label{thm:increment}
  Let $j, k \in \{\pm 1, \ldots, \pm M\}$ with $k\neq -j$.
  For every reduced $\psi = \theta_{j_\ell} \cdots \theta_{j_1} \in \Theta \setminus \{\Id\}$ with $j = j_\ell$,
  \begin{equation}
    \Delta_{j, k}^{\lb}(\zeta, \alpha) \le \Delta_k(\psi; \zeta, \alpha) \le \Delta_{j, k}^{\ub}(\zeta, \alpha),
    \label{eq:increment-bound}
  \end{equation}
  where, with $\theta_k'(z) = q_k^2/(1 - \conj{\delta_k}z)^2$,
  \begin{align}
    \Delta_{j, k}^{\lb}(\zeta, \alpha)
     & \defeq
    - \log \abs{\theta_k'\ab(z_k^{\fix})}
    + \min_{w \in C_j} \log \abs{R_k(w; \zeta)}
    + \min_{w \in C_j} \log \abs{R_k(w; \alpha)},
    \label{eq:deltajk-lb} \\
    \Delta_{j, k}^{\ub}(\zeta, \alpha)
     & \defeq
    - \log \abs{\theta_k'\ab(z_k^{\fix})}
    + \max_{w \in C_j} \log \abs{R_k(w; \zeta)}
    + \max_{w \in C_j} \log \abs{R_k(w; \alpha)}.
    \label{eq:deltajk-ub}
  \end{align}
\end{theorem}

The constants $\Delta_{j, k}^{\lb}$ and $\Delta_{j, k}^{\ub}$ are determined entirely by $\theta_j$, $\theta_k$, $\zeta$, and $\alpha$, completely independent of the tail $\theta_{j_{\ell - 1}} \cdots \theta_{j_1}$ of $\psi$.
Telescoping~\eqref{eq:increment-bound} along the prefix path $\psi_i = \theta_{j_i} \cdots \theta_{j_1}$ yields the prefix bound~\eqref{eq:prefix-lower} from~\cref{sec:graph-G}, and identifies the constants $\Delta_{j, k}^{\lb}$ as the edge weights of the graph $G$ used in \cref{alg:enumerate}.

\subsection{Closed form and precomputation}
\label{sec:edge-weights-closed}

The boundary extrema in~\eqref{eq:deltajk-lb}--\eqref{eq:deltajk-ub} admit a particularly simple closed form, because $R_k(w; \zeta)$ is a M\"obius transformation in~$w$.
As $w$ traces the circle $C_j$, the values $R_k(w; \zeta)$ fill the circle
\begin{equation*}
  \Sigma_{j, k}^{\zeta}
  \defeq \ab\{R_k(w; \zeta) : w \in C_j\}.
\end{equation*}
Let $\delta_{j, k}^{\zeta} \in \Complex$ and $q_{j, k}^{\zeta} > 0$ denote the center and radius of $\Sigma_{j, k}^{\zeta}$ (explicit formulas in \cref{sup:sec:max-min-on-circle}).
Because $\Sigma_{j, k}^{\zeta} = \ab\{w \in \Complex : \abs{w - \delta_{j, k}^{\zeta}} = q_{j, k}^{\zeta}\}$, the extrema of $\abs{w}$ over $w \in \Sigma_{j, k}^{\zeta}$ are attained on the ray through~$0$ and~$\delta_{j, k}^{\zeta}$.
With $\eta_{j, k}^{\zeta} = q_{j, k}^{\zeta}/\abs{\delta_{j, k}^{\zeta}}$ from~\eqref{sup:eq:eta-zeta} and $\eta_{j, k}^{\zeta} < 1$, we obtain
\begin{equation}
  \max_{w \in C_j}\abs{R_k(w; \zeta)}
  = \abs{\delta_{j, k}^{\zeta}} \ab(1 + \eta_{j, k}^{\zeta}),
  \qquad
  \min_{w \in C_j}\abs{R_k(w; \zeta)}
  = \abs{\delta_{j, k}^{\zeta}} \ab(1 - \eta_{j, k}^{\zeta}),
  \label{eq:max-min}
\end{equation}
with the same formulas for $\abs{R_k(w; \alpha)}$ using $(\delta_{j, k}^{\alpha}, \eta_{j, k}^{\alpha})$.
Equations \eqref{sup:eq:delta-q-zeta-full} and~\eqref{sup:eq:eta-zeta} provide explicit M\"obius-composition expressions in terms of $(\theta_j, \theta_k, \zeta, \alpha)$ for $(\delta_{j, k}^{\zeta}, q_{j, k}^{\zeta})$, $(\delta_{j, k}^{\alpha}, q_{j, k}^{\alpha})$, and hence for $(\eta_{j, k}^{\zeta}, \eta_{j, k}^{\alpha})$.

Therefore, substituting~\eqref{eq:max-min} into~\eqref{eq:deltajk-lb}--\eqref{eq:deltajk-ub} produces $\Delta_{j, k}^{\lb}$ and $\Delta_{j, k}^{\ub}$ in closed form.
Two consequences are central to the algorithm of \cref{sec:algorithm}:
\begin{itemize}
  \item The bounds depend only on the geometric data $\{(\delta_m, q_m)\}_{m = 1}^M$ of the holes and on $\zeta, \alpha \in \Domain$.
        They do not depend on the truncation parameter $U_{\max}$, on the word $\psi$, or on any other run-time quantity.
  \item All $2M(2M - 1)$ constants $\Delta_{j, k}^{\lb}, \Delta_{j, k}^{\ub}$ can therefore be precomputed in $O(M^2)$ time and memory, once and for all, before the enumeration starts.
\end{itemize}

\begin{remark}[Tightness of the increment bounds]
  \label{rem:increment-tightness}
  The gap $\Delta_{j, k}^{\ub} - \Delta_{j, k}^{\lb}$ measures the worst-case spread of $\Delta_k(\psi; \zeta, \alpha)$ as $\psi$ varies.
  Numerical experiments (\cref{fig:exp-increment-boxplot}) suggest that this gap is small in practice: the increment is effectively determined by the pair $(\theta_j, \theta_k)$ alone.
  \Cref{thm:increment} makes this dependence rigorous by supplying closed-form bounds that depend only on those two generators (together with $\zeta$ and~$\alpha$).
\end{remark}

% ============================================================
\section{Error bound and algorithmic complexity}
\label{sec:error}

We next bound the relative error of potential truncation and the cost of the enumeration algorithm \cref{alg:enumerate}.
Both reduce to estimating the same counting function,
\begin{equation*}
  N(u) \defeq \Card \ab\{\theta \in \Theta'': U_{\zeta, \alpha}(\theta) \le u\}.
\end{equation*}
We derive the asymptotic behavior of this function in \cref{sec:counting}, and then apply that single estimate to obtain the error bound in \cref{sec:error-statement} and the complexity bound in \cref{sec:complexity}.

\subsection{Counting estimate from Pollicott's theorem}
\label{sec:counting}

By~\eqref{eq:term-as-potential}, $\abs{R_{\zeta, \alpha}(\theta) - 1} = e^{-U_{\zeta, \alpha}(\theta)}$, and hence
\begin{equation*}
  \{\theta \in \Theta'': U_{\zeta, \alpha}(\theta) \le U_{\max}\}
  = \{\theta \in \Theta'': \abs{R_{\zeta, \alpha}(\theta) - 1} \ge e^{-U_{\max}}\}.
\end{equation*}
Applying~\eqref{eq:pollicott-count} with $T = e^{U_{\max}}$ yields the following estimate.
\begin{lemma}[Counting estimate for the potential]
  \label{lem:counting}
  Let $C > 0$ and $\delta \in (0, 2)$ be as in \cref{thm:pollicott}.
  Then,
  \begin{equation}
    N(u) \defeq
    \Card \ab\{\theta \in \Theta'' : U_{\zeta, \alpha}(\theta) \le u\}
    \sim C e^{\delta u},
    \qquad \text{as }u \to \infty.
    \label{eq:counting-U}
  \end{equation}
\end{lemma}

This lemma is the only counting input used in the following: \cref{thm:error} expresses the omitted tail as a Stieltjes integral against~$N$, and \cref{prop:alg} bounds the number of words visited when applying \cref{alg:enumerate}.
The exponent $\delta$ is the Hausdorff dimension of the limit set of~$\Theta$; numerical methods for approximating such dimensions for Kleinian groups go back at least to Jenkinson and Pollicott~\cite{jenkinson2002Calculating}.
In \cref{sec:experiments}, we estimate $\delta$ from the spectral radius of a transfer operator and compare the resulting rate $1 - \delta$ with observed relative errors.

\subsection{Asymptotic relative-error bound}
\label{sec:error-statement}

\begin{theorem}[Asymptotic relative-error bound]
  \label{thm:error}
  Let $\delta \in (0, 1)$ be the Hausdorff dimension of the limit set of $\Theta$ and let $C > 0$ be the constant of Pollicott's counting theorem (\cref{thm:pollicott}).
  Then,
  \begin{equation}
    \frac{\abs{\omega_{U_{\max}}(\zeta, \alpha) -
        \omega(\zeta, \alpha)}}
    {\abs{\omega(\zeta, \alpha)}}
    \le
    \ab(1 + o(1))
    \frac{C \delta}{1 - \delta} e^{-(1 - \delta)U_{\max}}
    \qquad \text{as } U_{\max} \to \infty.
    \label{eq:error-bound}
  \end{equation}
\end{theorem}

\begin{proof}[Sketch]
  A standard product-bound inequality (\cref{sup:sec:error-proof}) yields
  \begin{equation}
    \frac{\abs{\omega_{U_{\max}} - \omega}}{\abs{\omega}}
    \le
    \ab(1 + o(1))
    \sum_{U_{\zeta, \alpha}(\theta) > U_{\max}}\abs{R_{\zeta, \alpha}(\theta) - 1},
    \label{eq:err-as-tail}
  \end{equation}
  and so it suffices to estimate the tail sum.
  Writing this as a Stieltjes integral against $N$ from \cref{lem:counting} via~\eqref{eq:term-as-potential} and integrating by parts, we obtain
  \begin{equation}
    \sum_{U_{\zeta, \alpha}(\theta) > U_{\max}}\abs{R_{\zeta, \alpha}(\theta) - 1}
    =
    \int_{U_{\max}}^{\infty} e^{-u} \dl{N}(u)
    \sim
    \frac{C\delta}{1 - \delta} e^{-(1 - \delta)U_{\max}},
    \label{eq:tail-sum}
  \end{equation}
  where the asymptotic uses $N(u) \sim C e^{\delta u}$.
  Substituting~\eqref{eq:tail-sum} into~\eqref{eq:err-as-tail} yields~\eqref{eq:error-bound}.
\end{proof}

In words, the relative error of potential truncation decays exponentially in the cutoff~$U_{\max}$ at rate~$1 - \delta$, where $\delta$ is the Hausdorff dimension of the limit set of~$\Theta$.
Larger~$\delta$ corresponds to more crowded hole configurations, and therefore yields a slower guaranteed decay.

\subsection{Complexity of the enumeration algorithm}
\label{sec:complexity}

\begin{proposition}[Correctness and efficiency of Algorithm~\ref{alg:enumerate}]%
  {Correctness and efficiency of \cref{alg:enumerate}}
  \label{prop:alg}
  Suppose $G$ contains no negative cycle.
  Then, \cref{alg:enumerate}, with edge weights from~\eqref{eq:deltajk-lb} and the pruning rule of~\cref{sec:wmin}, returns every $\theta \in \Theta''_{\le U_{\max}}$ exactly once, and the number of words it visits is $O\ab(e^{\delta(U_{\max} - W_{\min})})$.
\end{proposition}

\begin{proof}
  \emph{Correctness.} Let $\theta = \theta_{j_\ell} \cdots \theta_{j_1} \in \Theta''_{\le U_{\max}}$, so that $U_{\zeta, \alpha}(\theta) \le U_{\max}$.
  Every ancestor $\psi$ of $\theta$ visited by the depth-first search is a prefix of the same reduced word, and hence~\eqref{eq:descendant-lower} gives $U_{\zeta, \alpha}(\theta) \ge U_{\zeta, \alpha}(\psi) + W_{\min}$.
  Therefore, $U_{\zeta, \alpha}(\psi) \le U_{\max} - W_{\min}$ and no ancestor is pruned, so $\theta$ is reached.
  Because $W_{\min} \le 0$, any visited word with $U_{\zeta, \alpha}(\theta) \le U_{\max}$ satisfies $U_{\zeta, \alpha}(\theta) \le U_{\max} - W_{\min}$ and is appended to~$\mathcal{O}$.
  Retaining one representative per inverse pair yields each element of $\Theta''_{\le U_{\max}}$ exactly once.
  \emph{Efficiency.} Every visited word lies in $\{\theta: U_{\zeta, \alpha}(\theta) \le U_{\max} - W_{\min}\}$, whose cardinality is $O\ab(e^{\delta(U_{\max} - W_{\min})})$ by~\eqref{eq:counting-U}.
\end{proof}

When all edge weights of $G$ are nonnegative, $W_{\min} = 0$ as in \cref{sec:wmin}, and the visit count is $O\ab(e^{\delta U_{\max}})$, matching $\Card \ab(\Theta''_{\le U_{\max}})$ up to a constant factor.
For negative weights, the extra factor $e^{- \delta W_{\min}}$ depends only on~$G$, not on~$U_{\max}$.

% ============================================================
\section{Numerical experiments}
\label{sec:experiments}

We test the increment bounds, the enumeration algorithm, and the relative-error analysis on three circular-hole configurations of increasing geometric difficulty (\cref{fig:exp-circles}), comparing potential truncation with classical level truncation.
\Cref{sec:exp-skprime} describes an external comparison with SKPrime~\cite{kropf2016SKPrime,crowdy2016Schottky}.
In \cref{sec:exp-heatmap},~$\zeta$ is varied over~$\Domain$ at fixed~$\alpha$ to compare the two truncations globally, rather than only for fixed~$(\zeta, \alpha)$.

\subsection{Experimental setup}
\label{sec:exp-setup}

\begin{itemize}
  \item \textbf{Case~1 (well separated):} $M = 2$ with centers $\delta_1 = -0.4$, $\delta_2 = 0.4$ and radii $q_1 = q_2 = 0.15$, far from each other and from the unit circle; evaluation at $\zeta = 0.5\iunit$, $\alpha = -0.5\iunit$ (\cref{fig:exp-case-1}).
  \item \textbf{Case~2 (closely packed):} $M = 2$ with centers $\delta_1 = -0.155$, $\delta_2 = 0.155$ and radii $q_1 = q_2 = 0.15$; evaluation at $\zeta = 0.5\iunit$, $\alpha = -0.5\iunit$ (\cref{fig:exp-case-2}).
  \item \textbf{Case~3 (three closely packed holes):} $M = 3$ with centers $\delta_1 = 0$, $\delta_2 = 0.575$, $\delta_3 = 0.34\iunit$ and radii $q_1 = 0.15$, $q_2 = 0.385$, $q_3 = 0.15$; evaluation at $\zeta = 0.6\iunit$, $\alpha = -0.5\iunit$ (\cref{fig:exp-case-3}).
\end{itemize}
Case~1 is a regime in which level truncation is already efficient.
Cases~2 and~3 are configurations for which short words need not produce small cross-ratio terms, so that level ordering correlates poorly with term size.
For each configuration, we compute a reference value $\omega_{\mathrm{ref}}$ through potential truncation at a fixed large cutoff $U_{\mathrm{ref}}$.
We take $U_{\mathrm{ref}} = 34$ in Cases~1 and~2, and set $U_{\mathrm{ref}} = 20$ in Case~3 because the number of retained factors grows rapidly with~$U_{\max}$.
Throughout this section, we report the relative error against~$\omega_{\mathrm{ref}}$ rather than the unavailable exact value~$\omega$,
\begin{equation}
  \text{relative error}
  = \frac{\abs{\omega_{\mathrm{approx}} - \omega_{\mathrm{ref}}}}{\abs{\omega_{\mathrm{ref}}}},
  \label{eq:rel-error}
\end{equation}
where $\omega_{\mathrm{approx}}$ is the approximate value returned by the truncation scheme under test.
Level truncation is swept over $L_{\max}$ and potential truncation over $U_{\max}$, and wall-clock times are recorded together with this relative error.

\begin{figure}[t]
  \centering
  \subfloat[Case~1]{\label{fig:exp-case-1}%
    \includegraphics[width=0.325\textwidth]{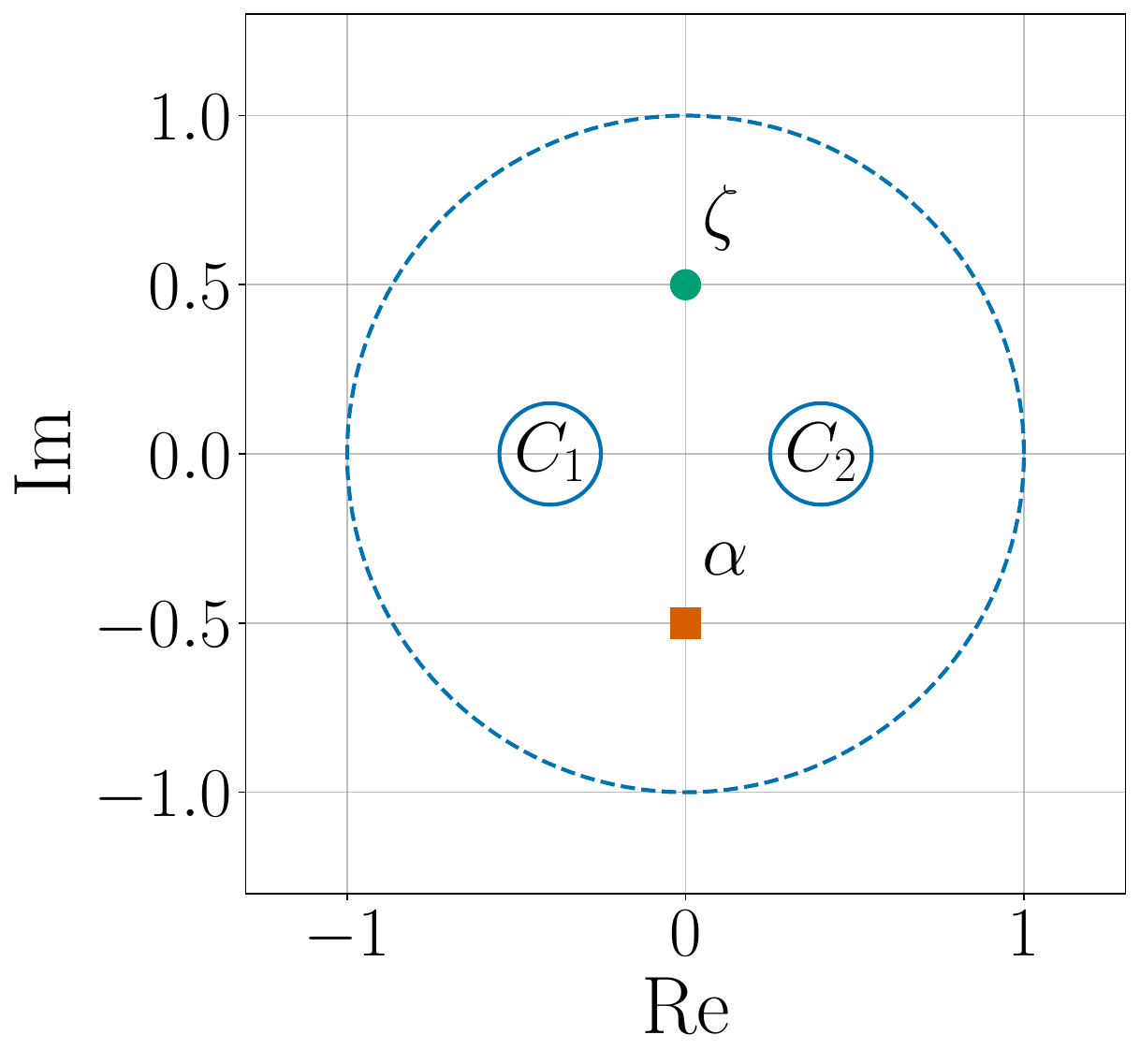}}%
  \hspace{0.01\textwidth}%
  \subfloat[Case~2]{\label{fig:exp-case-2}%
    \includegraphics[width=0.325\textwidth]{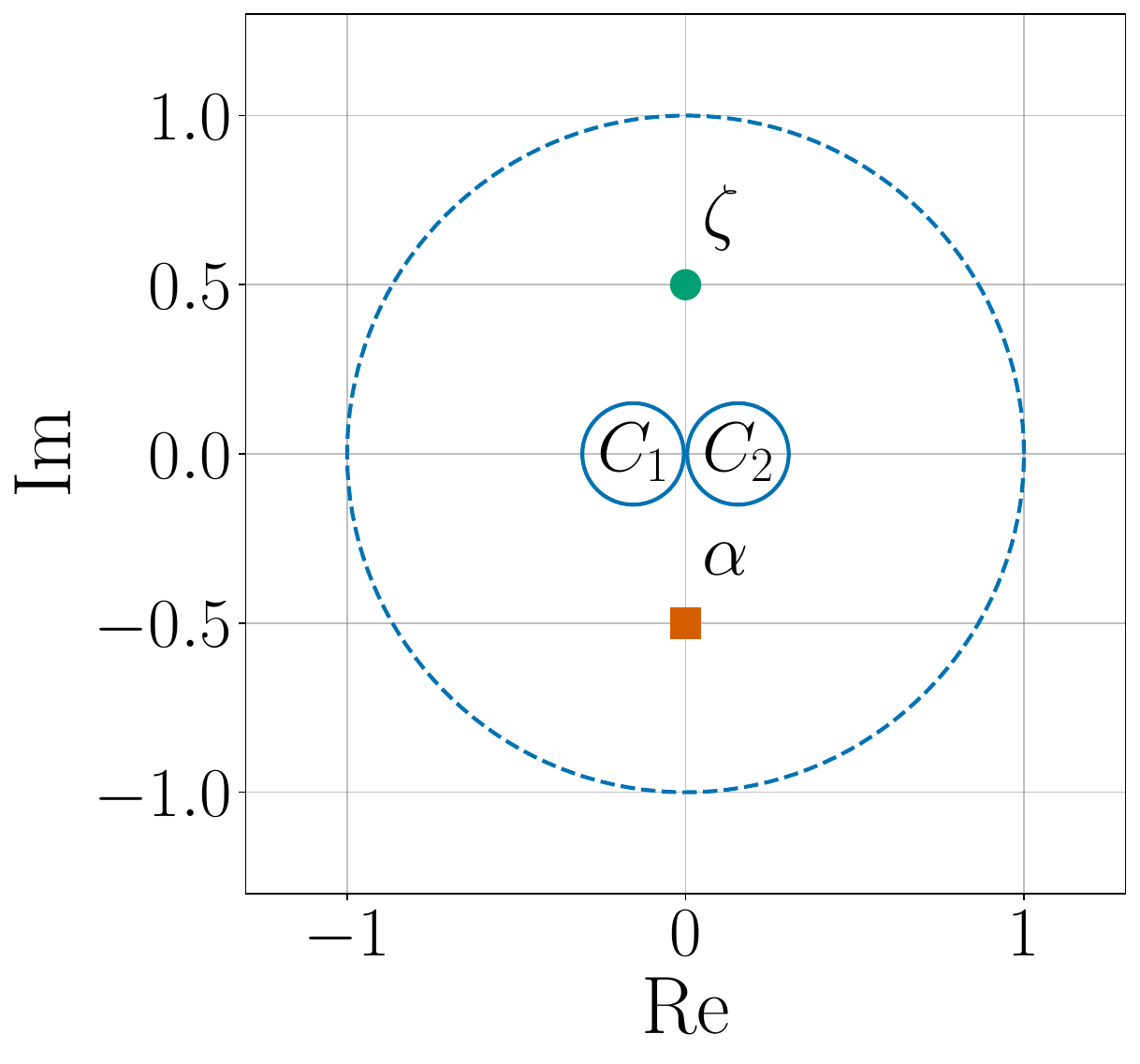}}%
  \hspace{0.01\textwidth}%
  \subfloat[Case~3]{\label{fig:exp-case-3}%
    \includegraphics[width=0.325\textwidth]{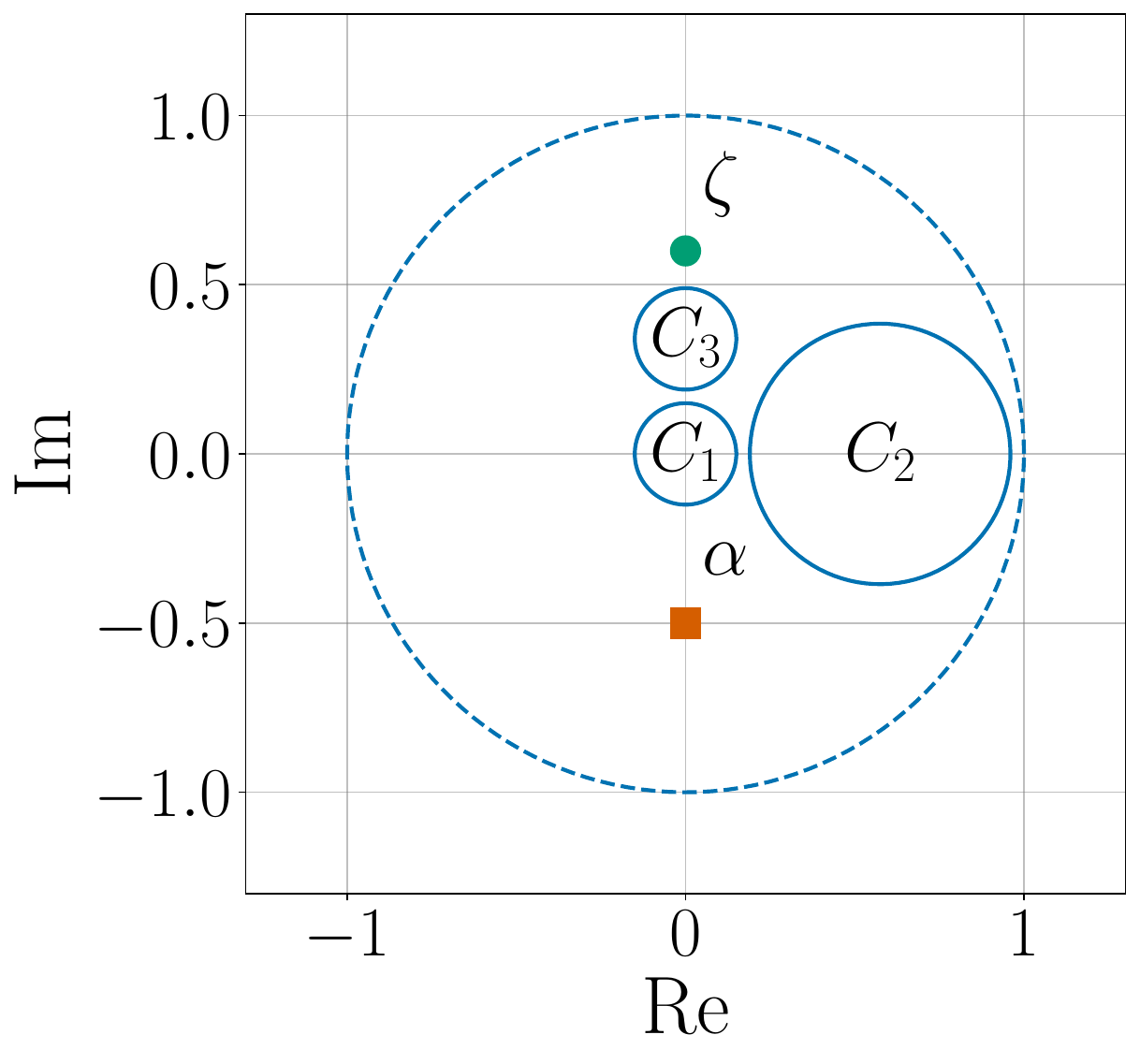}}
  \caption{Circular-hole configurations of increasing difficulty (evaluation points~$\zeta$,~$\alpha$ marked).
    Case~1: $\delta_1 = -0.4$, $\delta_2 = 0.4$, $q_1 = q_2 = 0.15$.
    Case~2: $\delta_1 = -0.155$, $\delta_2 = 0.155$, $q_1 = q_2 = 0.15$.
    Case~3: $\delta_1 = 0$, $\delta_2 = 0.575$, $\delta_3 = 0.34\iunit$, $q_1 = 0.15$, $q_2 = 0.385$, $q_3 = 0.15$.}
  \label{fig:exp-circles}
\end{figure}

\subsection{Uniform increment bounds}
\label{sec:exp-increment}

\begin{figure}[t]
  \centering
  \setlength{\abovecaptionskip}{4pt}%
  \includegraphics[width=\textwidth]{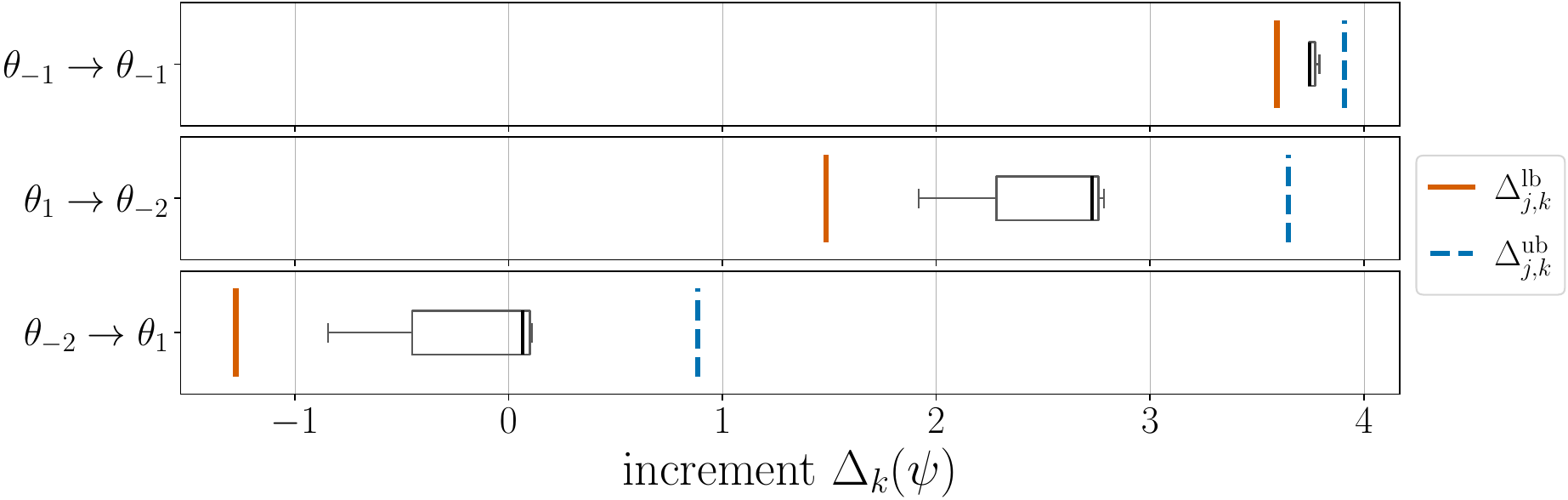}
  \caption{Potential increments $\Delta_k(\psi; \zeta, \alpha)$ in Case~2 over all reduced words with $\level(\theta) \le 7$.
    Panels: $(j, k) = (-1, -1)$, $(1, -2)$, $(-2, 1)$.
    Orange solid/blue dashed: uniform bounds $\Delta_{j, k}^{\lb}$ / $\Delta_{j, k}^{\ub}$.
    Numerical check of~\cref{thm:increment}: all samples lie within the bounds in every panel.}
  \label{fig:exp-increment-boxplot}
\end{figure}

\Cref{thm:increment} states that if a reduced word has leftmost factor $\theta_j$, then the increment $\Delta_k(\psi; \zeta, \alpha)$ obtained by prepending $\theta_k$ is bounded by constants that depend only on the pair $(j, k)$.
\Cref{fig:exp-increment-boxplot} checks this numerically on Case~2, in which consecutive generators interact strongly: the samples may be narrowly or broadly distributed and, for some pairs, negative.
The three panels show representative pairs: a narrow positive range for $(-1, -1)$, a wider positive range for $(1, -2)$, and a wide range including negative values for $(-2, 1)$.
In every panel, the samples lie between $\Delta_{j, k}^{\lb}$ and $\Delta_{j, k}^{\ub}$, and the spread within a pair is small compared with the gap between pairs.
This supports the use of the single lower bound $\Delta_{j, k}^{\lb}$ as the edge weight of~$G$ in \cref{sec:graph-G} and \cref{alg:enumerate}, in place of the $\psi$-dependent increment $\Delta_k(\psi; \zeta, \alpha)$.

\subsection{Pruning graph and relative error against runtime}
\label{sec:exp-results}

\begin{figure}[t]
  \centering
  \subfloat[Case~1]{\includegraphics[width=0.325\textwidth]{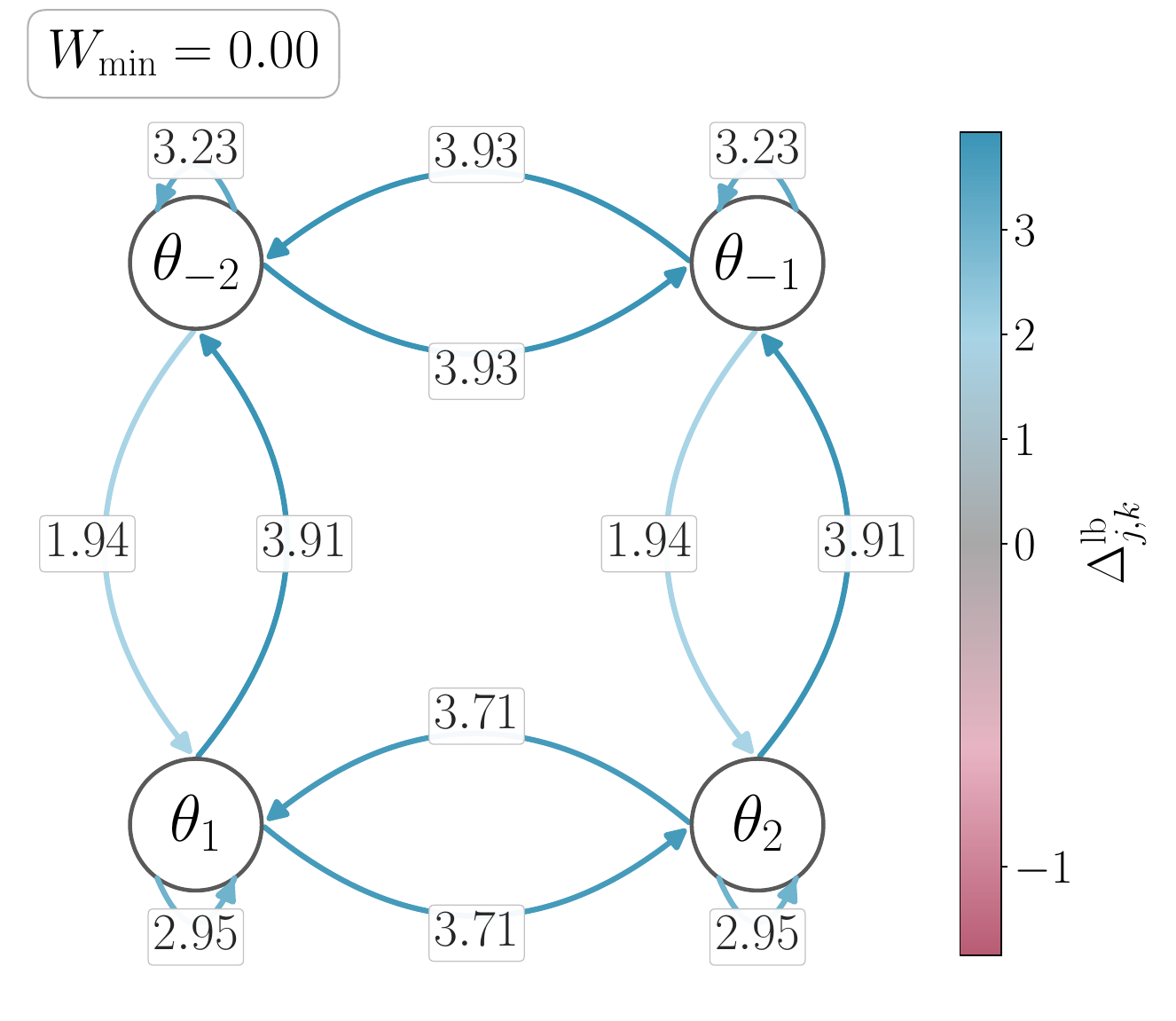}}%
  \hspace{0.01\textwidth}%
  \subfloat[Case~2]{\includegraphics[width=0.325\textwidth]{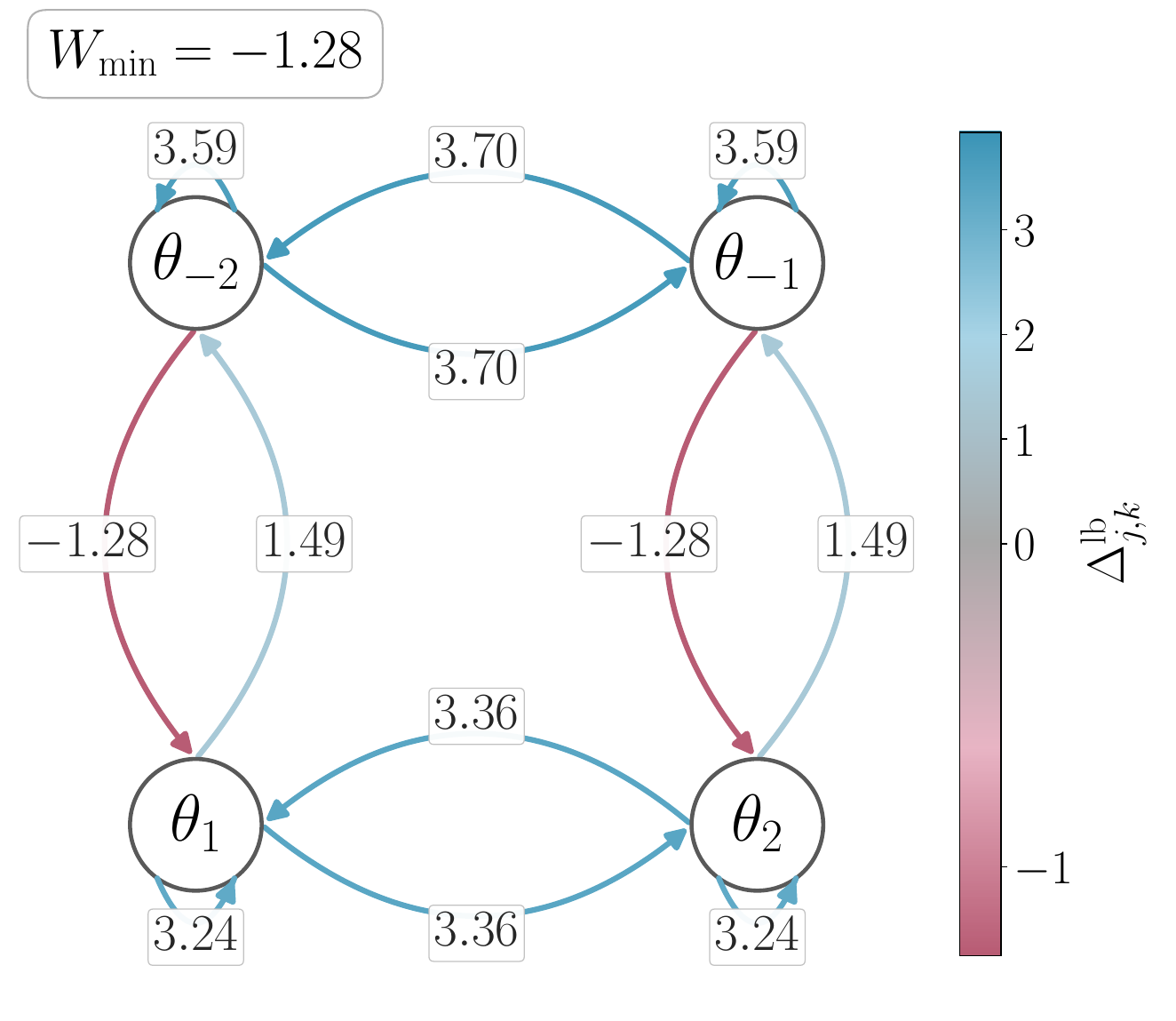}}%
  \hspace{0.01\textwidth}%
  \subfloat[Case~3]{\includegraphics[width=0.325\textwidth]{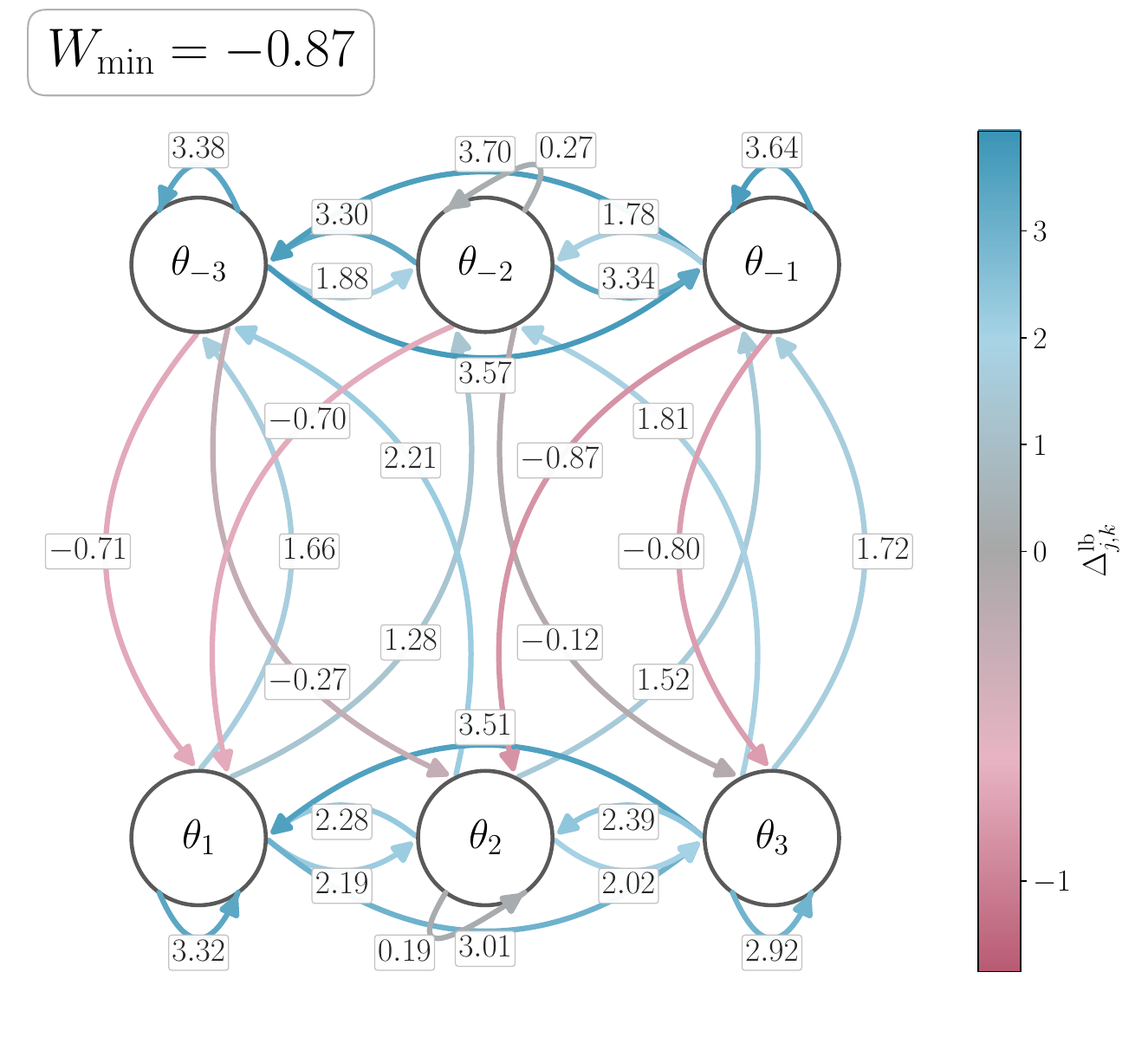}}
  \caption{Pruning graph~$G$ with edge weights $\Delta_{j, k}^{\lb}$ for Cases~1--3
    (blue: nonnegative; red: negative; shared color bar).
    Case~1: nonnegative weights only ($W_{\min} = 0$).
    Cases~2 and 3: negative edges present; $W_{\min}$ finite and no negative cycle in all three panels.}
  \label{fig:exp-pruning-graph}
\end{figure}

\Cref{fig:exp-pruning-graph} shows the edge weights used for subtree pruning in \cref{alg:enumerate}.
In Case~1, all displayed weights are nonnegative, so $W_{\min} = 0$ and the pruning threshold reduces to $U_{\max}$.
As the holes cluster, more edges become negative.
Case~3 has the densest set of negative weights.
In all three configurations, the Bellman--Ford algorithm yields a finite $W_{\min}$ and no negative cycle, which is the hypothesis of \cref{prop:alg} under which the algorithm is correct and visits $O\ab(e^{\delta(U_{\max} - W_{\min})})$ words.

\begin{figure}[t]
  \centering
  \setlength{\abovecaptionskip}{2pt}%
  \subfloat[Case~1]{\includegraphics[width=0.325\textwidth]{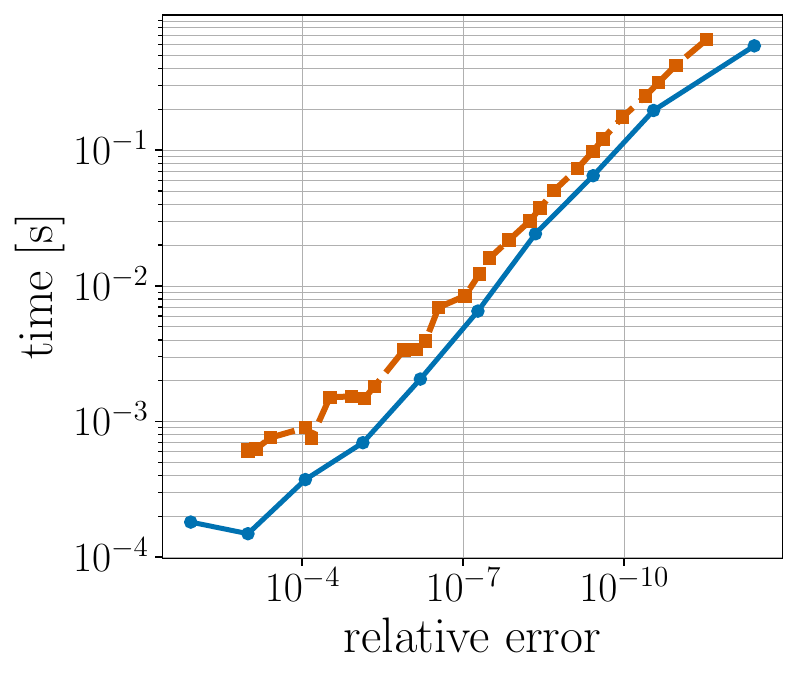}}%
  \hspace{0.01\textwidth}%
  \subfloat[Case~2]{\includegraphics[width=0.325\textwidth]{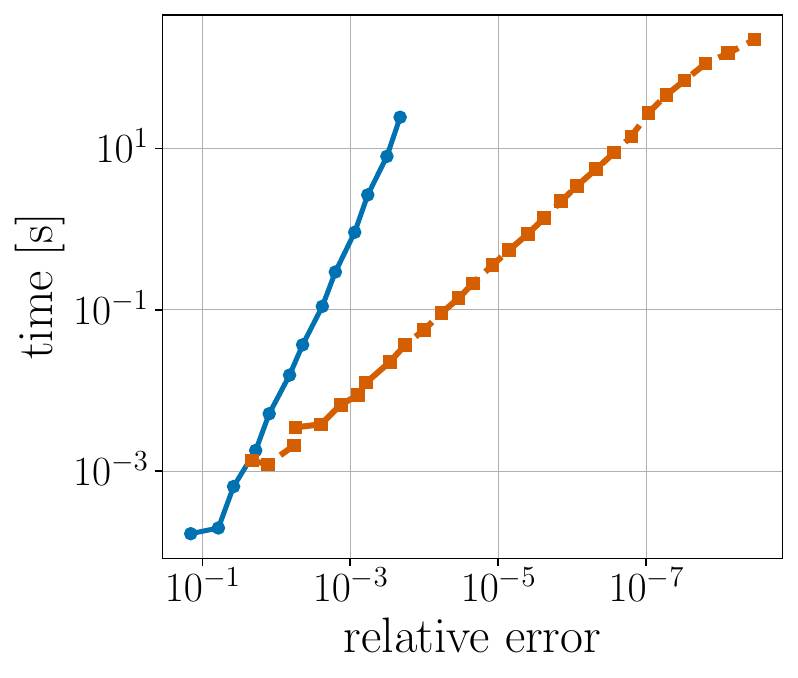}}%
  \hspace{0.01\textwidth}%
  \subfloat[Case~3]{\includegraphics[width=0.325\textwidth]{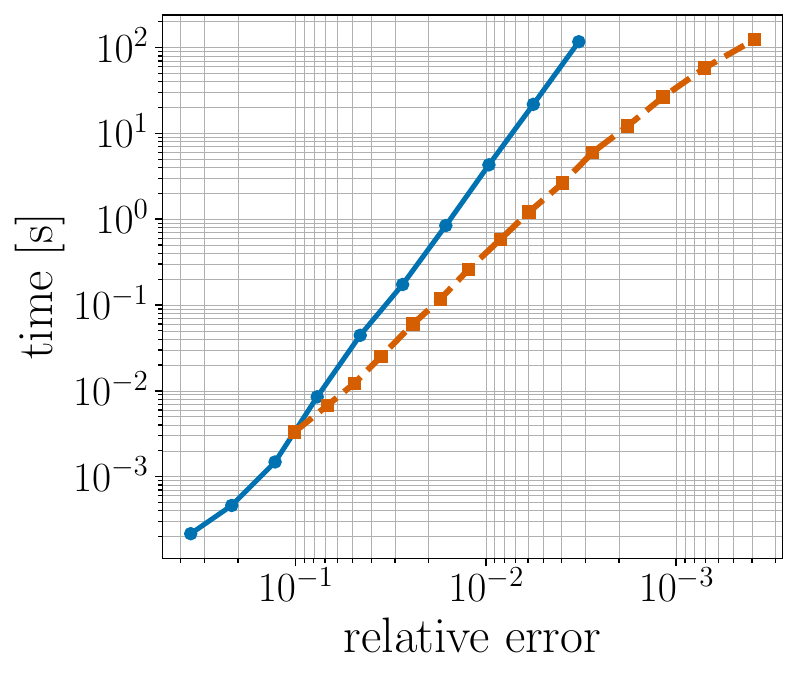}}\\[-0.15em]
  \includegraphics[width=0.58\textwidth]{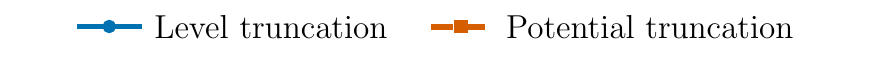}%
  \vspace{-0.6em}%
  \caption{Wall-clock time against relative error for level truncation and potential truncation (Cases~1--3).
    Each marker: one cutoff $L_{\max}$ or $U_{\max}$ (accuracy increases to the right).
    Case~1: the two methods are nearly comparable.
    Cases~2 and 3: potential truncation reaches the same error in substantially less time.}
  \label{fig:exp-runtime-error}
\end{figure}

\Cref{fig:exp-runtime-error} records the wall-clock time against relative error for both truncation schemes.
For potential truncation, each marker uses the product of every factor collected by the pruned depth-first search of~\cref{alg:enumerate}.
The visited set can include words with $U_{\zeta, \alpha}(\theta) > U_{\max}$, so the product can be more accurate at a fixed~$U_{\max}$.
In Case~1, both methods converge quickly and potential truncation is only marginally faster at a given accuracy, as expected when level ordering already tracks term size.
In Cases~2 and~3, which have larger limit-set dimensions than Case~1, potential truncation reaches the same relative error substantially faster.
Thus, the practical gain is clearest when level ordering correlates poorly with term size.

\subsection{Asymptotic error bound}
\label{sec:exp-asymptotic}

The empirical Hausdorff dimensions estimated from the spectral radius of the transfer operator are $\delta \approx 0.31$, $0.46$, and $0.77$ for Cases~1--3, following Jenkinson and Pollicott~\cite{jenkinson2002Calculating}.
The decay rates predicted by \cref{thm:error} are therefore $1 - \delta \approx 0.69$, $0.54$, and $0.23$, so Case~3 is expected to converge most slowly.

\begin{table}[t]
  \caption{Relative error of the partial product~\eqref{eq:trunc-def} and the asymptotic bound~\eqref{eq:error-bound} at selected~$U_{\max}$ (Cases~1--3).
    Here, $C$,~$\delta$ from $N(u)\sim C e^{\delta u}$; $N(U_{\max})$ is the number of retained factors.
    Cases~1 and 2: error tracks the bound.
    Case~3: slower decay; bound valid but conservative at moderate~$U_{\max}$.}
  \label{tab:exp-potential-error}
  \centering
  \footnotesize
  \setlength{\tabcolsep}{3.5pt}%
  \renewcommand{\arraystretch}{0.95}%
  % Cases 1--3: well_separated, closely_packed, three_closely_packed (data/20260620).
\begin{tabular}{@{}lllrrll@{}}
\toprule
Case & $C$ & $\delta$ & $U_{\max}$ & $N(U_{\max})$ & rel.\ error & asymp.\ bound \\
\midrule
Case~1 & $0.74$ & $0.31$ & $10$ & $16$ & $3.00\times 10^{-4}$ & $3.16\times 10^{-4}$ \\
 &  &  & $20$ & $324$ & $3.40\times 10^{-7}$ & $3.05\times 10^{-7}$ \\
 &  &  & $30$ & $6957$ & $2.93\times 10^{-10}$ & $2.95\times 10^{-10}$ \\
\addlinespace
Case~2 & $0.71$ & $0.46$ & $10$ & $63$ & $3.41\times 10^{-3}$ & $2.85\times 10^{-3}$ \\
 &  &  & $20$ & $7321$ & $1.37\times 10^{-5}$ & $1.33\times 10^{-5}$ \\
 &  &  & $30$ & $763223$ & $5.84\times 10^{-8}$ & $6.17\times 10^{-8}$ \\
\addlinespace
Case~3 & $0.57$ & $0.77$ & $10$ & $1166$ & $2.60\times 10^{-2}$ & $1.81\times 10^{-1}$ \\
 &  &  & $14$ & $25850$ & $6.25\times 10^{-3}$ & $7.09\times 10^{-2}$ \\
 &  &  & $18$ & $561496$ & $1.31\times 10^{-3}$ & $2.78\times 10^{-2}$ \\
\bottomrule
\end{tabular}

\end{table}

\Cref{tab:exp-potential-error} compares the observed relative error of the partial product~\eqref{eq:trunc-def} with the asymptotic bound~\eqref{eq:error-bound}.
With $C$ and $\delta$ fitted from $N(u)$, the last column is $\frac{C\delta}{1 - \delta} e^{-(1 - \delta)U_{\max}}$.
In Cases~1 and~2, the observed error tracks the bound closely over the tabulated range of~$U_{\max}$.
In Case~3, where $1 - \delta \approx 0.23$ is small, the bound remains an upper estimate, but is conservative at moderate~$U_{\max}$.

\subsection{Comparison with SKPrime}
\label{sec:exp-skprime}

\begin{table}[t]
  \caption{SKPrime (default setting) compared with potential truncation at selected~$U_{\max}$: Case~1 and a near unit circle configuration
    (left: geometry; $M = 2$, $\delta_1 = -0.3$, $\delta_2 = 0.75$, $q_1 = 0.15$, $q_2 = 0.24$).
    Potential: strict-product relative error and pruned-enumeration wall-clock time.
    SKPrime: single-point wall-clock time.
    Case~1: SKPrime already $\sim 10^{-12}$; potential truncation approaches this as $U_{\max}$ grows.
    Near unit circle: SKPrime remains at $\sim 10^{-4}$; potential truncation continues to smaller error at larger cost.}
  \label{tab:exp-skprime}
  \centering
  \begin{minipage}[c]{0.30\textwidth}
    \centering
    \includegraphics[width=\linewidth]{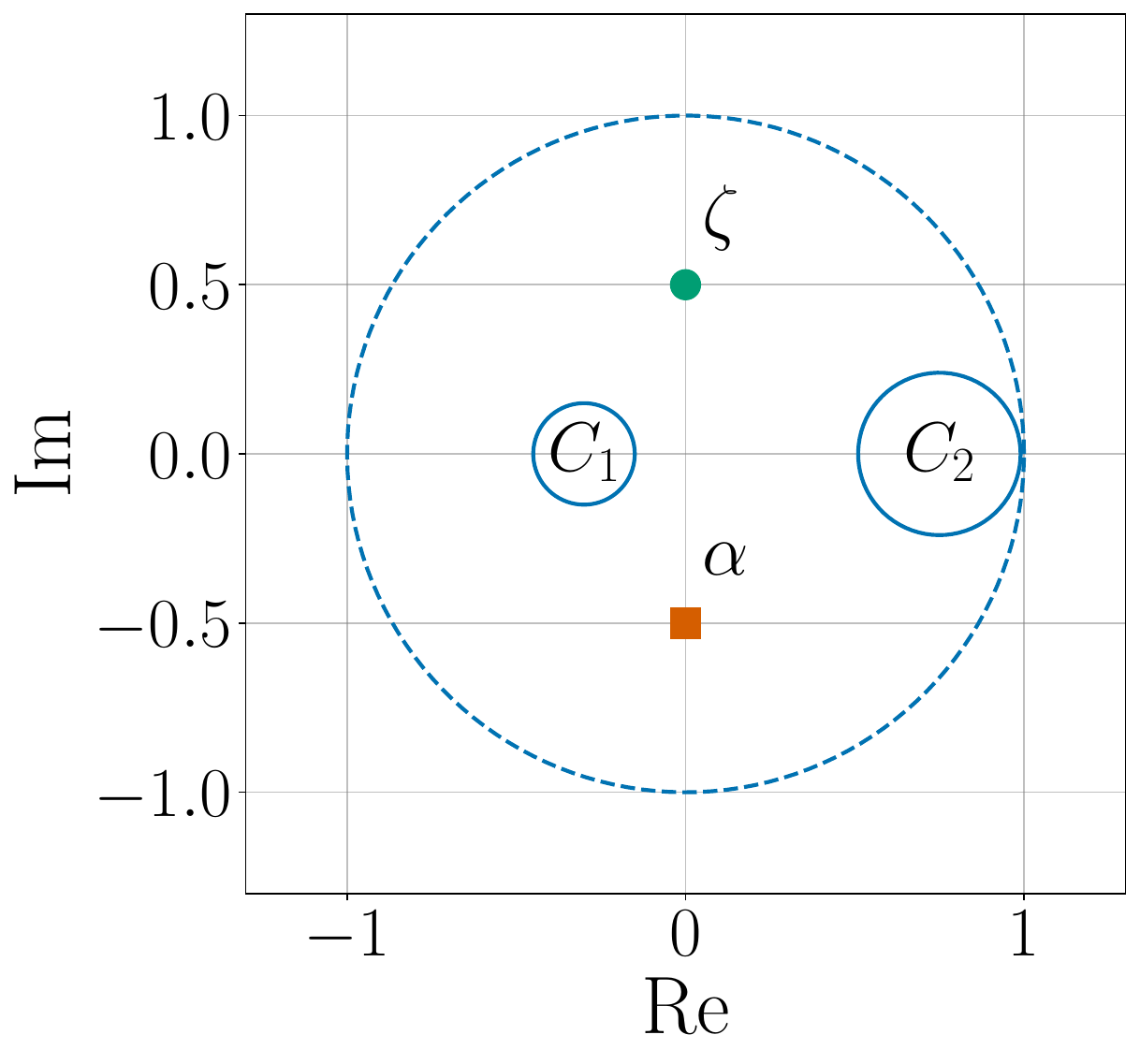}
  \end{minipage}%
  \hfill
  \begin{minipage}[c]{0.68\textwidth}
    \centering
    \footnotesize
    \setlength{\tabcolsep}{2.5pt}
    \resizebox{\linewidth}{!}{% SKPrime (default) vs potential truncation.
% Configs: well_separated (Case 1), near_boundary.
% Data stamp: 20260729
\begin{tabular}{lllrr}
\toprule
Config & Method & Truncation & Rel.\ error & Time (s) \\
\midrule
Case~1 & SKPrime & default setting & $3.89\times 10^{-12}$ & $0.154$ \\
 & Potential & $U_{\max} = 10$ & $3.00\times 10^{-4}$ & $1.51\times 10^{-3}$ \\
 & Potential & $U_{\max} = 20$ & $3.40\times 10^{-7}$ & $0.016$ \\
 & Potential & $U_{\max} = 30$ & $2.93\times 10^{-10}$ & $0.315$ \\
\addlinespace
Near unit circle & SKPrime & default setting & $3.66\times 10^{-4}$ & $0.225$ \\
 & Potential & $U_{\max} = 10$ & $1.18\times 10^{-3}$ & $5.12\times 10^{-3}$ \\
 & Potential & $U_{\max} = 20$ & $3.12\times 10^{-6}$ & $0.167$ \\
 & Potential & $U_{\max} = 30$ & $7.80\times 10^{-9}$ & $9.15$ \\
\bottomrule
\end{tabular}
}
  \end{minipage}
\end{table}

As an external baseline, we now compare against SKPrime~\cite{kropf2016SKPrime,crowdy2016Schottky}, which evaluates $\omega$ through boundary-value schemes on circular domains; see \cref{sec:intro-bvp}.
\Cref{tab:exp-skprime} lists the default SKPrime evaluation together with that of potential truncation at increasing~$U_{\max}$ for Case~1 and for a near-unit-circle configuration with one hole close to the unit circle.
The relative errors for both methods are measured against the same high-accuracy reference as in~\cref{sec:exp-setup}.
The advantage of potential truncation is the \emph{a priori} relative-error control of~\cref{thm:error}: a prescribed tolerance determines a sufficient~$U_{\max}$, and that accuracy is attained by taking the corresponding wall-clock time.
In Case~1, the default SKPrime error is already of order $10^{-12}$, and potential truncation approaches this regime as $U_{\max}$ grows.
In the near-unit-circle configuration, the default SKPrime error remains of order $10^{-4}$, whereas potential truncation continues to smaller errors for larger~$U_{\max}$ at greater cost.

\subsection{Spatial relative error for fixed \texorpdfstring{$\alpha$}{alpha}}
\label{sec:exp-heatmap}

The experiments described above use a fixed $(\zeta, \alpha)$ and vary the truncation parameters.
For this experiment, we instead fix~$\alpha$ and vary~$\zeta$ over~$\Domain$ to compare the global accuracy of level and potential truncation at a common term budget.
We take $M = 2$ with centers $\delta_1 = -0.5$, $\delta_2 = 0.7$ and radii $q_1 = q_2 = 0.2$, a geometry distinct from Cases~1--3, and set $\alpha = 0.2 - 0.4\iunit$.
At each sample, we form $\omega_{\mathrm{ref}}$ as in~\cref{sec:exp-setup}.
With $L_{\max}$ and $U_{\max}$ taken to be sufficiently large, we evaluate the relative error~\eqref{eq:rel-error} for both truncations using the partial products of the first $N = 1000$ factors in each ordering.

\begin{figure}[t]
  \centering
  \includegraphics[width=0.8\textwidth]{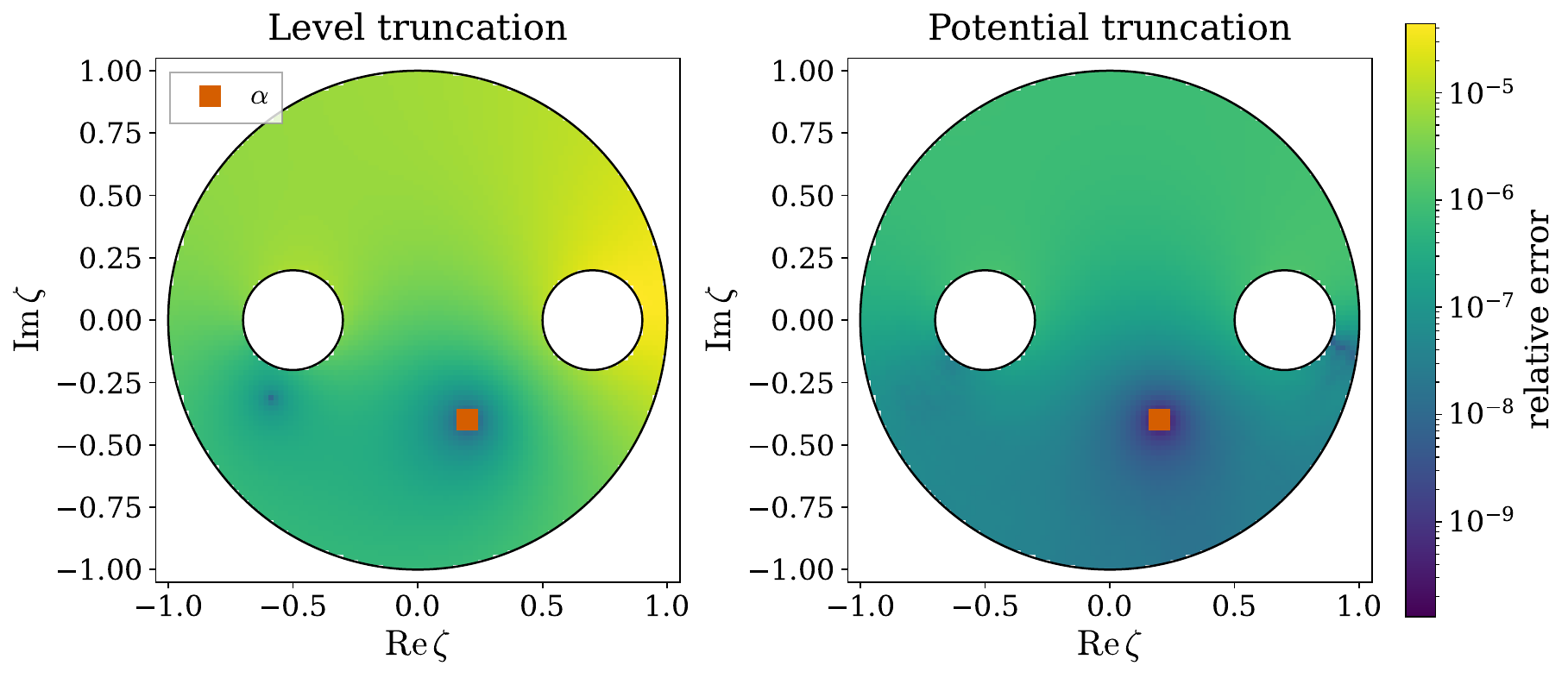}
  \caption{Relative error over $\zeta \in \Domain$ for fixed $\alpha = 0.2 - 0.4\iunit$
    ($M = 2$, $\delta_1 = -0.5$, $\delta_2 = 0.7$, $q_1 = q_2 = 0.2$; orange marker:~$\alpha$).
    Left: level truncation.
    Right: potential truncation.
    Both panels use the first $N = 1000$ factors in the corresponding ordering.
    Largest errors near holes and the unit circle; potential truncation more accurate throughout.}
  \label{fig:exp-heatmap}
\end{figure}

\Cref{fig:exp-heatmap} shows the resulting fields on a $100 \times 100$ grid over~$\Domain$.
Potential truncation is more accurate than level truncation throughout~$\Domain$, so the advantage seen for fixed~$(\zeta, \alpha)$ persists globally.
The largest errors of either method lie near the hole boundaries and the unit circle.

% ============================================================
\section{Conclusions and future work}
\label{sec:conclusion}

To truncate the infinite-product representation of the Schottky--Klein prime function, we proposed a potential truncation method based on the potential $U_{\zeta, \alpha}(\theta) \defeq - \log \abs{R_{\zeta, \alpha}(\theta) - 1}$, alongside a fast enumeration algorithm for its efficient evaluation.
\Cref{thm:increment} establishes uniform closed-form bounds on the potential increment, depending only on the prepended generator and the leftmost basic M\"obius transformation of the reduced word.
These bounds define the edge weights of the finite graph~$G$ that enables subtree pruning in the enumeration algorithm (\cref{alg:enumerate}).
\Cref{thm:error} provides an asymptotic relative-error bound with a decay rate of $1 - \delta$, and the same counting estimate governs the algorithmic complexity.
Numerical experiments confirm the method's theoretical properties---the absence of negative cycles in~$G$ and the tracking of the asymptotic error bound---while demonstrating superior efficiency and global accuracy over classical level truncation, especially in closely packed configurations.
Comparison with SKPrime~\cite{kropf2016SKPrime,crowdy2016Schottky} demonstrates that our approach provides rigorous \emph{a priori} error control, allowing the accuracy to be systematically improved by increasing the cutoff.

Several future directions remain, spanning algorithmic improvements, theoretical developments, and broader applications.
On the algorithmic side, sharpening the gap $\Delta_{j, k}^{\ub} - \Delta_{j, k}^{\lb}$ would tighten the edge weights of~$G$ and improve the pruning efficiency.
Additionally, characterizing negative cycles in~$G$ and developing a reliable fallback strategy when the pruning rule becomes vacuous are key practical challenges.
From a theoretical perspective, establishing a complete convergence theory for the infinite product---extending beyond classical geometric restrictions~\cite{baker1897Abels,belokolos1994Algebrogeometric,crowdy2007Computing} and our sufficient condition---remains an important goal, alongside investigating the dependence of the Hausdorff dimension $\delta$ on the hole geometry~\cite{jenkinson2002Calculating}.
Finally, extending the potential framework to secondary prime functions~\cite{vasconcelos2015Secondary} and to doubly periodic or multiply connected mapping problems~\cite{baddoo2019Periodic,krishnamurthy2023Nvortex} offers a promising direction for future research.

% -------------------- Acknowledgments --------------------
\section*{Acknowledgments}
The authors would like to thank Professor Takayasu Matsuo for his guidance and all the members of the laboratory for their support.
The authors also thank Professor Ken'ichiro Tanaka (Institute of Science Tokyo) for fruitful discussions and valuable advice.

% -------------------- Bibliography --------------------
\bibliographystyle{siamplain}
\bibliography{paper_skprime}

\end{document}

\maketitle

% ============================================================
\section{Computing the minimum walk weight by the Bellman--Ford algorithm}
\label{sup:sec:bellman-ford}

The minimum walk weight $W_{\min}$ of the finite weighted graph~$G$ in \cref{sec:wmin} is obtained by the standard Bellman--Ford algorithm, which also detects negative cycles.
The algorithm is stated here for completeness; \cref{alg:enumerate} uses the returned value~$W_{\min}$, or aborts if a negative cycle is detected.

\begin{algorithm}[t]
  \caption{Minimum walk weight on $G$ (Bellman--Ford).}
  \label{sup:alg:bellman-ford}
  \begin{algorithmic}[1]
    \STATE \textbf{Input:} lower bounds $\{\Delta_{j, k}^{\lb}\}$ of $G$, precomputed via~\eqref{eq:deltajk-lb}.
    \STATE \textbf{Output:} $W_{\min} \le 0$, or the flag \texttt{NegativeCycle}.
    \STATE $W[v]\Leftarrow 0$ for every $v \in V$.
    \FOR{$i = 1, \ldots, \abs{V} - 1$}
    \FORALL{$(\theta_j, \theta_k) \in E$}
    \STATE $W[\theta_k]\Leftarrow \min \ab(W[\theta_k], W[\theta_j] + \Delta_{j, k}^{\lb})$.
    \ENDFOR
    \ENDFOR
    \FORALL{$(\theta_j, \theta_k) \in E$}
    \IF{$W[\theta_j] + \Delta_{j, k}^{\lb} < W[\theta_k]$}
    \STATE \textbf{return} \texttt{NegativeCycle}.
    \ENDIF
    \ENDFOR
    \STATE \textbf{return} $W_{\min}\Leftarrow \min_{v \in V} W[v]$.
  \end{algorithmic}
\end{algorithm}

% ============================================================
\section{Detailed proofs of the cross-ratio decomposition}
\label{sup:sec:cross-ratio-proofs}

We complete \cref{sec:increment-why} by proving~\eqref{eq:increment-separated} and, in \cref{sup:sec:max-min-on-circle}, by computing $(\delta_{j, k}^{\zeta}, q_{j, k}^{\zeta})$ and $\eta_{j, k}^{\zeta} = q_{j, k}^{\zeta}/\abs{\delta_{j, k}^{\zeta}}$.

\subsection{From the definition to the two-image form}
\label{sup:sec:step1}

With $\zeta' = \psi(\zeta)$ and $\alpha' = \psi(\alpha)$ as in \cref{sec:increment-why}, we expand~\eqref{eq:increment-rewrite} using the standard identity
\begin{equation}
  R(a, b, c, d)^{-1} - 1
  = - \frac{(b - a)(d - c)}{(d - b)(c - a)}
  = -R(a, d, b, c).
  \label{sup:eq:cross-ratio-minus-one}
\end{equation}
Applying~\eqref{sup:eq:cross-ratio-minus-one} to the numerator and denominator of~\eqref{eq:increment-rewrite} with $(a, b, c, d) = (\zeta, \alpha, \zeta', \alpha')$ and $(\theta_k^{-1}(\zeta), \theta_k^{-1}(\alpha), \zeta', \alpha')$, respectively, the factor $(\alpha' - \zeta')$ cancels and we obtain
\begin{align}
  \Delta_k(\psi; \zeta, \alpha)
   & =
  - \log \abs{\frac{\theta_k^{-1}(\zeta) - \theta_k^{-1}(\alpha)}{\zeta - \alpha}}
  \notag
  \\
   & \quad
  + \log \abs{\frac{\zeta' - \theta_k^{-1}(\zeta)}{\zeta' - \zeta}}
  + \log \abs{\frac{\alpha' - \theta_k^{-1}(\alpha)}{\alpha' - \alpha}}.
  \label{sup:eq:three-logs}
\end{align}

\subsection{Separation using the fixed point}
\label{sup:sec:step2}

The first term in~\eqref{sup:eq:three-logs} couples $\zeta$ and $\alpha$ together.
We split it using the fixed point $z_k^{\fix}$ from~\eqref{eq:Rz-def}.

Write $\theta_k^{-1}(z) = (A'z + B')/(C'z + D')$.
For any M\"obius map $G(z) = (Az + B)/(Cz + D)$, a direct expansion gives
\begin{equation}
  \frac{G(u) - G(v)}{u - v} = \frac{AD - BC}{(Cu + D)(Cv + D)}
  \qquad\text{for all }u, v,
  \label{sup:eq:mobius-difference}
\end{equation}
and $G'(z) = (AD - BC)/(Cz + D)^2$.
Setting $G = \theta_k^{-1}$ and $u = z_k^{\fix}$ in~\eqref{sup:eq:mobius-difference} and rearranging yields
\begin{equation}
  \frac{1}{C'v + D'}
  =\frac{1}{(\theta_k^{-1})'\ab(z_k^{\fix})}\cdot\frac{z_k^{\fix} - \theta_k^{-1}(v)}{(z_k^{\fix} - v)(C'z_k^{\fix} + D')}.
  \label{sup:eq:fixed-point-form}
\end{equation}
Applying~\eqref{sup:eq:mobius-difference} with $G = \theta_k^{-1}$ and $(u, v) = (\zeta, \alpha)$, and substituting~\eqref{sup:eq:fixed-point-form} for $v = \zeta$ and for $v = \alpha$, gives
\begin{align*}
  \frac{\theta_k^{-1}(\zeta) - \theta_k^{-1}(\alpha)}{\zeta - \alpha}
   & = \frac{A'D' - B'C'}{(C'\zeta + D')(C'\alpha + D')},
  \\
  \frac{1}{C'\zeta + D'}
   & = \frac{1}{(\theta_k^{-1})'\ab(z_k^{\fix})}
  \cdot \frac{z_k^{\fix} - \theta_k^{-1}(\zeta)}{(z_k^{\fix} - \zeta)(C'z_k^{\fix} + D')},
  \\
  \frac{1}{C'\alpha + D'}
   & = \frac{1}{(\theta_k^{-1})'\ab(z_k^{\fix})}
  \cdot \frac{z_k^{\fix} - \theta_k^{-1}(\alpha)}{(z_k^{\fix} - \alpha)(C'z_k^{\fix} + D')}.
\end{align*}
Multiplying the three identities and using $(\theta_k^{-1})'(z_k^{\fix}) = (A'D' - B'C')/(C'z_k^{\fix} + D')^2$, we obtain
\begin{align*}
  \frac{\theta_k^{-1}(\zeta) - \theta_k^{-1}(\alpha)}{\zeta - \alpha}
   & = \frac{1}{(\theta_k^{-1})'\ab(z_k^{\fix})}
  \frac{z_k^{\fix} - \theta_k^{-1}(\zeta)}{z_k^{\fix} - \zeta}
  \frac{z_k^{\fix} - \theta_k^{-1}(\alpha)}{z_k^{\fix} - \alpha},
  \\
   & = \theta_k'\ab(z_k^{\fix})
  \frac{z_k^{\fix} - \theta_k^{-1}(\zeta)}{z_k^{\fix} - \zeta}
  \frac{z_k^{\fix} - \theta_k^{-1}(\alpha)}{z_k^{\fix} - \alpha},
\end{align*}
since $(\theta_k^{-1})'(z_k^{\fix}) = 1/\theta_k'(z_k^{\fix})$.
Hence
\begin{align*}
  \Delta_k(\psi; \zeta, \alpha)
   & =
  - \log \abs{\theta_k'\ab(z_k^{\fix})}
  - \log \abs{\frac{z_k^{\fix} - \theta_k^{-1}(\zeta)}{z_k^{\fix} - \zeta}}
  - \log \abs{\frac{z_k^{\fix} - \theta_k^{-1}(\alpha)}{z_k^{\fix} - \alpha}}
  \\
   & \quad
  + \log \abs{\frac{\zeta' - \theta_k^{-1}(\zeta)}{\zeta' - \zeta}}
  + \log \abs{\frac{\alpha' - \theta_k^{-1}(\alpha)}{\alpha' - \alpha}}.
\end{align*}
Grouping the $\zeta$- and $\alpha$-dependent terms and applying~\eqref{eq:R-cross-ratio} with~\eqref{eq:Rz-def} gives
\begin{align*}
  \log \abs{\frac{\zeta' - \theta_k^{-1}(\zeta)}{\zeta' - \zeta}}
  - \log \abs{\frac{z_k^{\fix} - \theta_k^{-1}(\zeta)}{z_k^{\fix} - \zeta}}
   & = \log \abs{
    R\ab(\zeta', z_k^{\fix}, \theta_k^{-1}(\zeta), \zeta)
  }
  \\
   & = \log \abs{R_k(\zeta'; \zeta)},
\end{align*}
and the same identity with $(\zeta, \zeta')$ replaced by $(\alpha, \alpha')$ gives $\log \abs{R_k(\alpha'; \alpha)}$.
Therefore
\begin{align*}
  \Delta_k(\psi; \zeta, \alpha)
   & =
  - \log \abs{\theta_k'\ab(z_k^{\fix})}
  + \log \abs{R_k(\zeta'; \zeta)}
  + \log \abs{R_k(\alpha'; \alpha)},
\end{align*}
which is~\eqref{eq:increment-separated}.

% ============================================================
\section{Computation of \texorpdfstring{$\max / \min |R_k(w; \zeta)|$ on $C_j$}{max/min |R\_k(w;zeta)| on C\_j}}
\label{sup:sec:max-min-on-circle}

We derive $(\delta_{j, k}^{\zeta}, q_{j, k}^{\zeta})$ and the $(\cdot, \alpha)$ analogues by pole inversion on~$C_j$.

\subsection{M\"obius image of a circle by pole inversion}
\label{sup:sec:pole-inversion}

Let $G(z) = (az + b)/(cz + d)$ with $ad - bc \neq 0$, and let $p = -d/c$ be its pole.
For a circle $\Gamma = \{z \in \Complex : \abs{z - \delta} = q\}$ with $\delta \in \Complex$ and $q > 0$, assume $p \notin \Gamma$, i.e.\ $\abs{p - \delta} \neq q$.
Let $\xi$ denote the inverse of~$p$ in~$\Gamma$.
Identity~\eqref{sup:eq:mobius-difference} applies to~$G$.

\begin{lemma}[M\"obius image of a circle]
  \label{sup:lem:mobius-circle-image}
  Under the hypotheses above, $G(\Gamma) = \{z \in \Complex : \abs{z - \delta'} = q'\}$, where $\delta'$ and~$q'$ denote the center and radius of~$G(\Gamma)$ and
  \begin{equation}
    \xi
    = \delta + \frac{q^2}{\conj{p - \delta}},
    \qquad
    \delta' = G(\xi),
    \qquad
    q'
    = \frac{\abs{ad - bc} q}
    {\abs{c}^2 \abs{q^2 - \abs{p - \delta}^2}}.
    \label{sup:eq:pole-inversion}
  \end{equation}
\end{lemma}

\begin{proof}
  Since $p \notin \Gamma$, the map $G$ is holomorphic on~$\Gamma$ and sends it to a circle (not a line).
  By definition of circle inversion, $(\xi - \delta)\conj{(p - \delta)} = q^2$, so $p$, $\delta$, and~$\xi$ are collinear and
  \begin{equation}
    \abs{\xi - \delta} = \frac{q^2}{\abs{p - \delta}},
    \qquad
    \abs{\xi - p} = \frac{\abs{q^2 - \abs{p - \delta}^2}}{\abs{p - \delta}}.
    \label{sup:eq:inverse-point-distances}
  \end{equation}
  A standard M\"obius fact gives $\delta' = G(\xi)$ as the center of~$G(\Gamma)$.
  Pick any $z_0 \in \Gamma$; then $q' = \abs{G(z_0) - G(\xi)}$.
  By~\eqref{sup:eq:mobius-difference} with $G$ in place of~$\theta_k$,
  \begin{equation*}
    q'
    = \frac{\abs{ad - bc} \abs{z_0 - \xi}}
    {\abs{c z_0 + d} \abs{c \xi + d}}
    = \frac{\abs{ad - bc} \abs{z_0 - \xi}}
    {\abs{c}^2 \abs{z_0 - p} \abs{\xi - p}}.
  \end{equation*}
  On~$\Gamma$ the Apollonius ratio $\abs{z_0 - \xi}/\abs{z_0 - p} = q/\abs{p - \delta}$ is constant for $z_0 \in \Gamma$.
  Substituting the Apollonius ratio and the formula for $\abs{\xi - p}$ in~\eqref{sup:eq:inverse-point-distances} into the expression for~$q'$, and cancelling the factor $\abs{z_0 - p}$, yields~\eqref{sup:eq:pole-inversion}.
\end{proof}

\subsection{Center and radius of \texorpdfstring{$\Sigma_{j, k}^{\zeta}$}{Sigma\_j,k^zeta}}
\label{sup:sec:apply-to-Rz}

Fix $j, k \in \{\pm 1, \ldots, \pm M\}$ with $k \neq -j$.
For $R_k(w; \zeta)$ from~\eqref{eq:Rz-def}, write
\begin{equation}
  R_k(w; \zeta)
  = \kappa_{k}^{\zeta} \frac{w - \theta_k^{-1}(\zeta)}{w - \zeta},
  \qquad
  \kappa_{k}^{\zeta}
  \defeq
  \frac{z_k^{\fix} - \zeta}{z_k^{\fix} - \theta_k^{-1}(\zeta)},
  \label{sup:eq:R-w-zeta}
\end{equation}
by~\eqref{eq:R-cross-ratio}.
The prefactor $\kappa_{k}^{\zeta}$ is a constant determined by $\theta_k$ and $\zeta$ alone (through $z_k^{\fix}$ and $\theta_k^{-1}(\zeta)$ in~\eqref{eq:Rz-def}).
Writing $R_k(w; \zeta) = (aw + b)/(cw + d)$ gives
\begin{equation}
  a = \kappa_{k}^{\zeta}, \
  b = - \kappa_{k}^{\zeta} \theta_k^{-1}(\zeta), \
  c = 1, \
  d = - \zeta,
  \qquad
  ad - bc = \kappa_{k}^{\zeta}\ab(\theta_k^{-1}(\zeta) - \zeta).
  \label{sup:eq:R-coeffs-z}
\end{equation}

We apply \cref{sup:lem:mobius-circle-image} to $w \mapsto R_k(w; \zeta)$ on $C_j$, identifying the image center and radius with $\delta_{j, k}^{\zeta}$ and~$q_{j, k}^{\zeta}$ from \cref{sec:edge-weights-closed}.
The pole is $p = \zeta$; with $\Gamma = C_j$, the inversion $\xi$ from \cref{sup:lem:mobius-circle-image} is denoted
\begin{equation}
  \xi_{j}^{\zeta}
  \defeq \delta_j + \frac{q_j^2}{\conj{(\zeta - \delta_j)}}.
  \label{sup:eq:xi-j-zeta}
\end{equation}
The image circle $\Sigma_{j, k}^{\zeta}$ has center and radius
\begin{equation}
  \delta_{j, k}^{\zeta}
  = R_k(\xi_{j}^{\zeta}; \zeta)
  = \frac{a \xi_{j}^{\zeta} + b}{c \xi_{j}^{\zeta} + d},
  \qquad
  q_{j, k}^{\zeta}
  = \frac{\abs{ad - bc} q_j}
  {\abs{c}^2 \abs{q_j^2 - \abs{\zeta - \delta_j}^2}},
  \label{sup:eq:delta-q-zeta}
\end{equation}
with $a, b, c, d$ as in~\eqref{sup:eq:R-coeffs-z}.
Substituting~\eqref{sup:eq:xi-j-zeta} and~\eqref{sup:eq:R-coeffs-z} into~\eqref{sup:eq:delta-q-zeta} yields
\begin{subequations}
  \label{sup:eq:delta-q-zeta-full}
  \begin{align}
    \delta_{j, k}^{\zeta}
     & = \frac{\kappa_{k}^{\zeta}(\xi_{j}^{\zeta} - \theta_k^{-1}(\zeta))}{\xi_{j}^{\zeta} - \zeta}
    = \frac{
      \kappa_{k}^{\zeta} [\conj{(\zeta - \delta_j)}(\delta_j - \theta_k^{-1}(\zeta)) + q_j^2]
    }{
      q_j^2 - \abs{\zeta - \delta_j}^2
    },
    \\
    q_{j, k}^{\zeta}
     & = \frac{\abs{\kappa_{k}^{\zeta}(\theta_k^{-1}(\zeta) - \zeta)} q_j}{\abs{q_j^2 - \abs{\zeta - \delta_j}^2}}.
  \end{align}
\end{subequations}
Set $\eta_{j, k}^{\zeta} \defeq q_{j, k}^{\zeta}/\abs{\delta_{j, k}^{\zeta}}$.
Substituting~\eqref{sup:eq:delta-q-zeta-full} gives
\begin{equation}
  \eta_{j, k}^{\zeta}
  = \frac{\abs{\theta_k^{-1}(\zeta) - \zeta} q_j}{\abs{\conj{(\zeta - \delta_j)}(\delta_j - \theta_k^{-1}(\zeta)) + q_j^2}}.
  \label{sup:eq:eta-zeta}
\end{equation}
Since $\zeta, \theta_k^{-1}(\zeta) \notin D_j$, the points $0$ and $\infty$ lie in the exterior of the disk bounded by $\Sigma_{j, k}^{\zeta}$, so $\abs{\delta_{j, k}^{\zeta}} > q_{j, k}^{\zeta}$; hence $0 \leq \eta_{j, k}^{\zeta} < 1$.
The extrema in~\eqref{eq:max-min} therefore read $\max_{w \in C_j}\abs{R_k(w; \zeta)} = \abs{\delta_{j, k}^{\zeta}} \ab(1 + \eta_{j, k}^{\zeta})$ and $\min_{w \in C_j}\abs{R_k(w; \zeta)} = \abs{\delta_{j, k}^{\zeta}} \ab(1 - \eta_{j, k}^{\zeta})$.
Replacing $\zeta$ by $\alpha$ gives $(\delta_{j, k}^{\alpha}, q_{j, k}^{\alpha}, \eta_{j, k}^{\alpha})$.

% ============================================================
\section{Detailed proof of the asymptotic relative-error bound}
\label{sup:sec:error-proof}

We prove~\eqref{eq:err-as-tail} and~\eqref{eq:tail-sum} from the proof sketch of \cref{thm:error}.

\subsection{Reduction to the tail sum}
\label{sup:sec:reduction-to-tail}

Set $\epsilon_\theta \defeq R_{\zeta, \alpha}(\theta) - 1$ and
\begin{equation}
  s \defeq \sum_{\theta \in \Theta''_{> U_{\max}}}\abs{\epsilon_\theta},
  \qquad
  \Theta''_{> U_{\max}}
  \defeq \{\theta \in \Theta'': U_{\zeta, \alpha}(\theta) > U_{\max}\}.
\end{equation}
By~\eqref{eq:omega-as-perturbation} and~\eqref{eq:trunc-def},
\begin{equation}
  \frac{\abs{\omega_{U_{\max}} - \omega}}{\abs{\omega}}
  =\frac{\abs{\prod_{\theta \in \Theta''_{> U_{\max}}}(1 + \epsilon_\theta) - 1}}
  {\abs{\prod_{\theta \in \Theta''_{> U_{\max}}}(1 + \epsilon_\theta)}}.
  \label{sup:eq:err-quotient}
\end{equation}
For the numerator in~\eqref{sup:eq:err-quotient}, we expand $\prod_{\theta \in \Theta''_{> U_{\max}}}(1 + \epsilon_\theta) - 1$ as a multilinear sum and apply the triangle inequality to obtain $\abs{\prod(1 + \epsilon_\theta) - 1} \le \prod \ab(1 + \abs{\epsilon_\theta}) - 1$.
Since $1 + t \le e^{t}$ for $t \ge 0$, we have $\prod \ab(1 + \abs{\epsilon_\theta}) \le e^{s}$, hence
\begin{equation}
  \abs{\prod_{\theta \in \Theta''_{> U_{\max}}}(1 + \epsilon_\theta) - 1}
  \le e^{s} - 1.
  \label{sup:eq:tail-product-bound}
\end{equation}
Combining~\eqref{sup:eq:err-quotient} and~\eqref{sup:eq:tail-product-bound}, and using $\abs{\prod(1 + \epsilon_\theta)} \le e^{s} \to 1$ as $U_{\max} \to \infty$, we obtain
\begin{equation}
  \frac{\abs{\omega_{U_{\max}} - \omega}}{\abs{\omega}}
  \le \frac{e^{s} - 1}{\abs{\prod_{\theta \in \Theta''_{> U_{\max}}}(1 + \epsilon_\theta)}}
  =\ab(1 + o(1))s,
  \label{sup:eq:err-as-tail-detailed}
\end{equation}
as $s \to 0$, because $e^{s} - 1 = \ab(1 + o(1))s$.
Hence
\begin{equation*}
  \frac{\abs{\omega_{U_{\max}} - \omega}}{\abs{\omega}}
  \le
  \ab(1 + o(1))
  \sum_{U_{\zeta, \alpha}(\theta) > U_{\max}}\abs{\epsilon_\theta},
\end{equation*}
which is~\eqref{eq:err-as-tail}.

\subsection{Stieltjes integration by parts}
\label{sup:sec:stieltjes}

By~\eqref{eq:counting-U} and $\abs{\epsilon_\theta} = e^{-U_{\zeta, \alpha}(\theta)}$,
\begin{equation}
  s = \sum_{U_{\zeta, \alpha}(\theta) > U_{\max}}\abs{\epsilon_\theta}
  =\int_{U_{\max}}^{\infty} e^{-u} \dl{N}(u).
  \label{sup:eq:as-stieltjes}
\end{equation}
Since $N$ is a right-continuous step function with unit jumps at the values $U_{\zeta, \alpha}(\theta)$, integration by parts with the smooth factor $e^{-u}$ gives
\begin{subequations}
  \label{sup:eq:ibp}
  \begin{align}
    \int_{U_{\max}}^{\infty} e^{-u} \dl{N}(u)
     & = \ab[N(u) e^{-u}]_{U_{\max}}^{\infty}
    +\int_{U_{\max}}^{\infty}N(u) e^{-u} \dl{u}
    \\
     & =-N(U_{\max}) e^{-U_{\max}}
    +\int_{U_{\max}}^{\infty}N(u) e^{-u} \dl{u},
  \end{align}
\end{subequations}
where the boundary term at $u = \infty$ vanishes since $N(u) e^{-u} \to 0$ ($\delta < 1$).
Substituting $N(u) \sim C e^{\delta u}$ and integrating,
\begin{subequations}
  \label{sup:eq:ibp-asymptotics}
  \begin{align}
    \int_{U_{\max}}^{\infty}N(u) e^{-u} \dl{u}
     & \sim C\int_{U_{\max}}^{\infty} e^{-(1 - \delta)u} \dl{u}
    =\frac{C}{1 - \delta} e^{-(1 - \delta)U_{\max}},
    \\
    N(U_{\max}) e^{-U_{\max}}
     & \sim C e^{-(1 - \delta)U_{\max}}.
  \end{align}
\end{subequations}
Combining the two asymptotics in~\eqref{sup:eq:ibp-asymptotics} gives
\begin{equation}
  s
  \sim \frac{C\delta}{1 - \delta} e^{-(1 - \delta)U_{\max}},
  \label{sup:eq:tail-sum-final}
\end{equation}
which is~\eqref{eq:tail-sum}.
Combining~\eqref{sup:eq:err-as-tail-detailed} with~\eqref{sup:eq:tail-sum-final} yields
\begin{equation*}
  \frac{\abs{\omega_{U_{\max}}(\zeta, \alpha) -
      \omega(\zeta, \alpha)}}
  {\abs{\omega(\zeta, \alpha)}}
  \le
  \ab(1 + o(1))
  \frac{C\delta}{1 - \delta} e^{-(1 - \delta)U_{\max}},
  \qquad \text{as } U_{\max} \to \infty,
\end{equation*}
which is~\eqref{eq:error-bound}.

% \bibliographystyle{siamplain}
% \bibliography{paper_skprime}